\documentclass{article}

\usepackage{amssymb}
\usepackage{amsfonts}
\usepackage{amsmath}
\usepackage{amsthm}
\usepackage{rotating}
\usepackage{color}
\usepackage[pdftex,bookmarks,colorlinks]{hyperref} 
\hypersetup{colorlinks,%
                    citecolor=blue,%
                      filecolor=blue,%
                      linkcolor=red,%
                      urlcolor=blue}%
\usepackage[pdftex]{hyperref}

\newtheorem{theorem}{Theorem}
\newtheorem{lemma}{Lemma}[section]

\newtheorem{remark}[theorem]{Remark}
\newtheorem{corollary}[lemma]{Corollary}

\newtheorem{definition}[theorem]{Definition}
\newtheorem{problem}[theorem]{Problem}

\newcommand{\cc}{\mathcal{C}}
\newcommand{\SL}{{\rm{SL}}}
\newcommand{\GL}{{\rm{GL}}}

\newcommand{\gl}{{\rm{GL}}}

\newcommand{\Sp}{{\rm{Sp}}}
\newcommand{\psp}{{\rm{PSp}}}
\newcommand{\So}{{\rm{SO}}}

\newcommand{\Og}{{\rm{O}}}
\newcommand{\Su}{{\rm{SU}}}

\newcommand{\gu}{{\rm{GU}}}

\newcommand{\ppd}{{\rm{ppd}}}
\newcommand{\Prob}{{\rm{Prob}}}

\newcommand{\cnqsl}{\SL, n, q}
\newcommand{\cnqsp}{\Sp, n, q}
\newcommand{\cnqsu}{\Su, n, q}
\newcommand{\cnqsog}{\So^\eps, n, q}
\newcommand{\cnqso}{\So^+, n, q}

\newcommand{\cnq}{\Phi^X(n, q)}
\newcommand{\cnqx}{X, n, q}

\newcommand{\cnqqsp}{\Phi^{\Sp}(n,q)}

\newcommand{\eps}{\varepsilon}
\newcommand{\la}{\langle}
\newcommand{\ra}{\rangle}

\title{Generation of finite classical groups by pairs of elements with large fixed point spaces \footnote{This research forms part of Discovery Project DP110101153 of the first and second authors funded by the Australian Research Council.}}

\author{Cheryl E. Praeger\footnote{Address: Centre for the Mathematics of Symmetry and Computation,
School  of Mathematics and Statistics M019,
The University of Western Australia,
35 Stirling Highway,
Crawley, WA 6009,
Australia.
Also affiliated with: King Abdulaziz University,
Jeddah,
Saudi Arabia. Email: cheryl.praeger@uwa.edu.au}, \'{A}kos Seress\thanks{Our colleague \'{A}kos Seress worked 
with us intensively on this research project for several years before his death, including the writing of several 
incomplete drafts of the paper. We completed the writing saddened by his passing. We hope he would be 
pleased at the outcome.} and  \c{S}\"{u}kr\"{u} Yal\c{c}\i nkaya\footnote{Address: Nesin Mathematics Village, \c{S}irince, \.{I}zmir, Turkey. Email: sukru.yalcinkaya@gmail.com.}}

\date{}

\begin{document}
\maketitle 

\begin{abstract}
We study `good elements' in finite $2n$-dimensional classical groups $G$: namely $t$
is a `good element' if $o(t)$ is divisible by a primitive 
prime divisor of $q^n-1$ for the relevant field order $q$, and $t$ fixes pointwise an $n$-space.
The group $\SL_{2n}(q)$ contains such elements,
and they are present in $\Su_{2n}(q), \Sp_{2n}(q), \So^\eps_{2n}(q)$, only if $n$ is odd, 
even, even, respectively. We
prove that there is an absolute positive constant $c$ such that two random conjugates of $t$
generate $G$ with probability at least $c$, if $G\ne \Sp_{2n}(q)$ with $q$ even. In the 
exceptional case $G=\Sp_{2n}(q)$ with $q$ even, two conjugates of $t$ never generate $G$: in this case
we prove that two random conjugates of $t$ generate a subgroup $\So^\eps_{2n}(q)$  
with probability at least $c$. 
The results (proved for all field orders at least $4$) underpin analysis of new constructive 
recognition algorithms for classical groups in even characteristic, which succeed where methods 
utilising involution centralisers are not available. 
\end{abstract}

\section{Introduction}
\label{introduction}

Motivated by an algorithmic application \cite{NS}, we address a problem in statistical 
group theory related to generating finite classical groups. It involves 
the notion of a `good element' in a classical group, defined as follows. 
Assume that $q=p^e$ is a prime power and $G=X_{2n}(q)$ is a $2n$-dimensional 
classical group where $X\in \{\SL, \Sp, \Su, \So^\eps \}$, acting naturally 
on a vector space $V \cong \mathbb{F}_{q^\delta}^{2n}$, where $\delta=2$ in 
the unitary case and $\delta=1$ otherwise. A prime number $r$ is a 
\emph{primitive prime divisor} if $q^n-1$ 
if $r \mid q^n-1$ but $r$ does not divide $q^i-1$ for any $i<n$. 
We often call $r$ a  \emph{$\ppd(n;q)$ prime}. By 
\cite{zsigmondy92.265}, $\ppd(n;q)$ primes exist except when $(n,q)= 
(6,2)$ or $(2,p)$, with $p$ a Mersenne prime. 

\begin{definition}\label{def:good}{\rm
Define $\Phi^X(2,q):=\{(q+1)/(q-1,2)\}$, and for $n\geq3$, let $\cnq$ be the set of all integers $m$ such that 
\begin{itemize}
\item[(i)] $m$ is divisible by some $\ppd(n;q^\delta)$ prime, or $(n,q^\delta)=(6,2)$ 
and $m$ is divisible by $9$, and
\item[(ii)] $m$ divides $\left\{ \begin{array}{ll}
(q^n-1)/(q-1) & \mbox{ if } X=\SL.\\
(q^n+1)/(q+1) & \mbox{ if } X=\Su \\
q^{n/2}+1 & \mbox{ if } X=\Sp \mbox{ or } \So^\eps.\\ 
                         \end{array} \right.$
\end{itemize}
We say that $t \in G=X_{2n}(q)$ is a \emph{good element} if $o(t)\in\Phi^X(n,q)$ and
$t$ has an $n$-dimensional fixed point space.
}
\end{definition}

We see in Lemma~\ref{good:cent:order} that each good element acts irreducibly on an $n$-dimensional complement to its fixed point space. 
Good elements arise algorithmically in new procedures~\cite{DLLO,NS} to recognise finite $d$-dimensional 
classical groups in even characteristic given as matrix groups in arbitrary representations. 
Previously such algorithms were only available in odd characteristic. 
At a certain point in these algorithms an element $t$ has been constructed with a large fixed point subspace,
and acting irreducibly on a complementary subspace, say $W$. The object (see \cite[p.234]{DLLO}) is to find a random conjugate $t'$ of 
$t$ such that $t, t'$ generate a classical group  on a space $W'$ of dimension $2\dim(W)$ and fix pointwise a 
complement of $W'$. A suitable element is then found in $\la t,t'\ra$ with the same properties as $t$. Used recursively, this 
plays a key role in constructing a small dimensional classical subgroup in approximately
$\log_2 d$ steps. To the knowledge of the first and third authors, the idea of using good elements to `double the degree' 
in this way was first conceived by the second author, and he developed this with Max Neunh\"offer into a full 
recognition algorithm which is available in the recog package in {\sf GAP}~\cite{recog} and will be 
described in a forthcoming paper~\cite{NS}.  

We consider a fundamental problem distilled from this process: namely, given a good element $t$ in a 
finite classical group, determine how likely it is that $t$ together with 
a random conjugate generate the whole classical group. A broad brush statement of our results is the following.

\begin{theorem}\label{broadbrush}
Let $t$ be a good element in $X_{2n}(q)$, where $X\in\{\SL, \Su, \Sp, \So^\eps\}$ and if 
$X=\SL$ then $q\geq4$, and let $t'$ be a uniformly distributed random conjugate of $t$. Then 
there is an absolute constant $c>0$ such that,
\begin{enumerate}
 \item[(a)]  $\la t,t^g\ra=X_{2n}(q)$ with probability greater than $c$, if $(X,q)\ne(\Sp, even)$, and
 \item[(b)]  $\la t,t^g\ra\cong \So^\eps_{2n}(q)$ with probability greater than $c$, if $(X,q)=(\Sp, even)$.
\end{enumerate}
\end{theorem}

We believe that this result also holds for $\SL_{2n}(q)$ with $q=2, 3$, but our analysis
in Section~\ref{c1} is not sufficiently strong to prove it. For an explanatory comment on the 
need to treat $(X,q)=(\Sp, even)$ separately, see Remark~\ref{rem:main}~(b).
The precise problem we address is the following.

\begin{problem}\label{problem}
Given a good element $t \in G=X_{2n}(q)$ and a uniformly distributed random 
conjugate $t^g$ of $t$, estimate from above the following probabilities.
\begin{itemize}
\item[(i)] $p_1(\cnqx):=\Prob(\,\langle t, t^g \rangle \mbox{ is reducible on $V$}\, )$;
\item[(ii)]
$p(\cnqx):=\Prob(\,\langle t, t^g \rangle  \ne G\ \mbox{and}\ \langle t, t^g \rangle \mbox{\ is 
irreducible}\, )$, where if $X=\Sp$ we assume that $q$ is odd;
\item[(iii)] 
if $X=\Sp$ and $q$ is even, ${\tilde{p}}(\cnqsp):=\Prob(\langle t, t^g\rangle  
\mbox{ is irreducible and lies }$ $\mbox{in a maximal subgroup $M$ of $\So_{2n}^\pm(q)$} \,)$ (cf. Lemma \ref{spso:even}).
\end{itemize}
\end{problem}

\begin{remark}\label{rem:parity}{\rm
First we note that, for $X\ne\SL$,  $n$ must satisfy certain parity restrictions in
order for $X_{2n}(q)$ to contain good elements, namely, such elements exist 
if and only if  $(X,n)= (\Su, \text{odd}), (\Sp, \text{even})$ or $(\So^\eps, \text{even})$ 
(see Lemma~\ref{good:cent:order} and Table~\ref{order:cent:good}). For the rest of the paper we assume that
these parity restrictions hold. Our main results are upper estimates for the quantities in Problem~\ref{problem}.

}
\end{remark}

\begin{theorem}\label{main}
Let $G=X_{2n}(q)$, where $X\in\{\SL, \Su, \Sp, \So^\eps\}$, $q\geq 2$, and $n\geq 2$ if $X=\SL, \Su$ or $\Sp$, and 
$n\geq4$ if $X=\So^\eps$. Then 
\begin{align*}
p_1(\cnqx) &< \begin{cases}
\frac{2}{q-1} - \frac{4}{q^{n+1}} &\mbox{if $X=\SL$}\\ 
\frac{1}{q^2-1} + \frac{3}{2q^{n+2}} &\mbox{if $X=\Su$}\\ 
\frac{1}{q-1} - \frac{1}{q^{n+1}} &\mbox{if $X=\Sp$}\\ 
\frac{1}{(2,q-1)q} +\frac{1}{q(q-1)} + \frac{3}{q^{n/2+1}} &\mbox{if $X=\So^+$}\\ 
\frac{1}{(2,q-1)q} +\frac{1}{q(q-1)} + \frac{4}{q^{n/2+1}} &\mbox{if $X=\So^-$.}\\   
\end{cases}
\end{align*}
\end{theorem}

\begin{theorem}\label{main2}
\begin{align*}
p(\cnqx) &< \begin{cases}
10.5\, q^{-n^2+n+2} + O(q^{-n^2})       &\mbox{if $X=\SL$}\\ 
9\, q^{-4(n^2-n)/3+2} + O(q^{-4n^2/3}) &\mbox{if $X=\Su$}\\ 
 3.7\, q^{-n^2/2} + O(q^{-2(n^2-n)/3})&\mbox{if $X=\Sp$}\\ 
 10.6\, q^{-n^2/2} + O(q^{-2(n^2-n)/3})&\mbox{if $X=\So^\eps$}\\ 
\end{cases}
\end{align*}
and, for $q$ even,  $\tilde{p}(\cnqsp)<   3.7\, q^{-n^2/2} + O(q^{-2(n^2-n)/3})$.
\end{theorem}

\begin{remark}\label{rem:main}{\rm 
(a) We note that our notation suppresses the order $o(t)$ of $t$ and 
its $G$-conjugacy class, since our estimates
are uniform across all $G$-conjugacy classes of good elements.

(b) Our reason for treating the case $X=\Sp$, $q$ even, separately in Problem~\ref{problem}~(ii) and (iii) is that, for 
these groups, if $\langle t, t^g \rangle\ne G$, then
this subgroup is either reducible or contained in a subgroup $\So_{2n}^\pm(q)$ (Lemma~\ref{spso:even}).

(c) The bounds in Theorem~\ref{main} are valid for all $n$ in the statement, and also for all $q$. 
However for type $X=\SL$, the upper bound is greater than 1 in the case $q=2$, and very close to $1$ for $q=3$. 
We were unable to obtain useful upper bounds for these cases. (See Section~\ref{c1}.) 

(d) Although we state asymptotic upper bounds in Theorem~\ref{main2}, explicit (but rather complicated)
upper bounds are proved. An explicit upper bound for each case may be obtained by adding the entries in the 
appropriate column of Table~\ref{tab:proof}. The proof strategy, discussed 
in Subsection~\ref{sec:strategy}, involves considering separately several families of maximal subgroups
which could contain $\la t,t^g\ra$. In all but the last family our results give the correct order of 
magnitude for the probability contribution, and hold for all $n\geq3$ (and sometimes also for $n=2$). 
Analysis for the last subgroup family $\mathbf{C}_9$ is less precise; it uses an upper bound \cite{hasa}
due to H\"as\"a on the number of conjugacy classes of these subgroups. 
}
\end{remark}

\subsection{Proof strategy}\label{sec:strategy}

Let $G=X_{2n}(q)$ as in Theorems~\ref{main} and~\ref{main2}, and let $t, t^g$ be as in Problem~\ref{problem}. 
If $\langle t, t^g \rangle \ne G$ then $t$ and $t^g$ belong to some maximal subgroup
$M<G$. By Aschbacher's theorem \cite{aschbacher84.469}, maximal subgroups of $G$ 
belong to one of nine categories $\mathbf{C}_1, \ldots , \mathbf{C}_9$.
Our standard reference is the book of Kleidman and Liebeck~\cite{kleidman1990}, where the exact 
conditions are given on when a subgroup belonging to a category $\mathbf{C}_i$ is maximal in $G$. 
We use a modified definition of $\mathbf{C}_1$, namely \emph{we take $\mathbf{C}_1$ to be the set of all
subgroups $M$ of $G$ that are maximal with respect to being reducible on $V$.}

Let $\mathcal{C}$ be a conjugacy class in $G$ of good elements. For a fixed $t\in\mathcal{C}$, set
\begin{align*}
W_1&=\{t^g  \mid g \in G, \, \exists M \in \mathbf{C}_1, \, t,t^g \in M  \},\\
W_2&=\{t^g \mid g \in G, \, \exists M \in \mathbf{C}_i, i\neq 1, \, t,t^g \in M, \mbox{and $\langle t,t^g\rangle$ irreducible} \}.
\end{align*}
Then 
$p_1(\cnqx) =  \frac{|W_1|}{|\mathcal{C}|} $
and 
${p}(\cnqx) =\frac{|W_2|}{|\mathcal{C}|} $.
We also set
\begin{align*}
\widehat{W_1}&=\{(t,t^g) \mid 
g \in G, \, \exists M\in \mathbf{C}_1, \, t,t^g \in M\cap\mathcal{C}  \},\\
\widehat{W_2}&=\{(t,t^g) \mid 
g \in G, \, \exists M \in \mathbf{C}_i, i\neq 1, \, t,t^g \in M \cap\mathcal{C}, \  
\mbox{and $\langle t,t^g\rangle$ irreducible} \}.
\end{align*}
Since $|W_i|$, $i=1,2$, is independent of the choice of $t \in \mathcal{C}$, it follows that $|\widehat{W_i}|=
|\mathcal{C}|\cdot |W_i|$ for each $i=1,2$. Hence
\begin{equation}\label{incl:p1}
p_1(\cnqx) =  \frac{|\widehat{W}_1|}{|\mathcal{C}|^2}\quad\mbox{and}\quad 
{p}(\cnqx) =  \frac{|\widehat{W}_2|}{|\mathcal{C}|^2}.
\end{equation}
Note that
\begin{equation}\label{second:ineq}
|\widehat{W}_1| \leq \sum_{M \in \mathbf{C}_1} |M\cap\mathcal{C}|^2
\quad\mbox{   and   }\quad
|\widehat{W}_2|\leq \sum_{i=2}^9\sum_{M \in \mathbf{C}_i} |M\cap \mathcal{C}|^2.
\end{equation}

\noindent
For each $i$, we identify a set  $\mathcal{S}_i$ of conjugacy classes of subgroups in $\mathbf{C}_i$ that cover all the possible 
subgroups $\la t,t^g\ra$. This set may be smaller than the set of all conjugacy classes in $\mathbf{C}_i$, see for example the analysis
for $\mathbf{C}_1$ in Section~\ref{c1}. For each $\mathbf{S} \in \mathcal{S}_i$, 
let $M(\mathbf{S})$ denote a representative subgroup of $\mathbf{S}$. Then (\ref{incl:p1}) and (\ref{second:ineq}) imply
\begin{equation}\label{third:prob}
p_1(\cnqx) \leq \sum_{\mathbf{S} \in \mathcal{S}_1} \frac{|G|}{|N_G(M(\mathbf{S}))|}\frac{|M(\mathbf{S})\cap \mathcal{C}|^2}{|\mathcal{C}|^2} 
\ \mbox{and}\ {p}(\cnqx) \leq \sum_{i=2}^9p_i(\cnqx)
\end{equation}
noting that some subgroups in $\mathbf{C}_1$ may not be self-normalising, where each
\begin{equation}\label{prob:cc}
p_i(\cnqx)=\sum_{\mathbf{S} \in \mathcal{S}_i} \frac{|G|}{|M(\mathbf{S})|}\frac{|M(\mathbf{S})\cap \mathcal{C}|^2}{|\mathcal{C}|^2}.
\end{equation}

For $G=\Sp_{2n}(q)$ with $q$ even, note that, if $\la t,t^g\ra$ is irreducible,
then it lies in a subgroup $\So^\eps_{2n}(q)$ of $G$ (see Remark~\ref{rem:main} and Lemma~\ref{spso:even}).
We need to count pairs $(t,t^g)$ such that $\la t,t^g\ra$ is irreducible and properly contained in 
such a subgroup. Thus there exists a subgroup $H\cong\So^\eps_{2n}(q)$ of $G$ such that 
$\langle t,t^g\rangle$ is contained in one of the maximal subgroups $M$ of $\cup_{i=2}^9\mathbf{C}^\eps_i(H)$,
where $\mathbf{C}^\eps_i(H)$ is the $i^{th}$ Aschbacher class of maximal subgroups of $H$, as above.

For each $\eps=\pm$, there is a single $G$-conjugacy class of subgroups $H$ by \cite[Prop. 4.8.6]{kleidman1990}, and hence each 
$H$-conjugacy class in $\mathbf{C}^\eps_i(H)$ lies in a single $G$-conjugacy class. The subgroups in
$\mathbf{C}^\eps_i(H)$ may be contained in maximal subgroups of $G$ lying in the Acshbacher class 
$\mathbf{C}_i$ for $G$. However this is not always the case. For example, the stabiliser 
in $H$ of an orthogonal decomposition of $V$ as a sum of $1$-dimensional spaces lies in 
$\mathbf{C}^\eps_2(H)$ but no subgroup of $\mathbf{C}_2$ for $G$ contains it, see Lemma~\ref{c2cases}.
Thus we define $\mathbf{C}_i'$ as the union of $\mathbf{C}_i$, for $G$, together with those subgroups
which are conjugate to a maximal subgroup in $\mathbf{C}^\eps_2(H)$ (for $\eps=\pm$) but which
do not lie in subgroups in $\mathbf{C}_i$. Similarly to our previous strategy, we identify 
a set  $\mathcal{S}_i'$ of conjugacy classes of subgroups in $\mathbf{C}_i'$ that cover all the irreducible 
$\la t,t^g\ra$ contained in subgroups in $\mathbf{C}^\eps_2(H)$ (for some $\eps=\pm$),
and for each $\mathbf{S} \in \mathcal{S}_i'$, we choose 
a representative subgroup $M(\mathbf{S})$ of $\mathbf{S}$. Then similar arguments 
to those above show that
\begin{equation}\label{main3:prob}
 \tilde{p}(\cnqsp)\leq\sum_{i=2}^9 \tilde{p}_i(\cnqsp),\ \mbox{where}\  \tilde{p}_i(\cnqsp)
=\sum_{\mathbf{S} \in \mathcal{S}_i'} \frac{|G|}{|M(\mathbf{S})|}\frac{|M(\mathbf{S})\cap \mathcal{C}|^2}{|\mathcal{C}|^2}.
\end{equation}

We prove some preliminary arithmetic results in Section~\ref{sec:pre}, and basic 
facts about good elements in Section~\ref{good:classical}. Section~\ref{ansn}
contains an analysis of the alternating and symmetric groups acting on their
deleted permutation modules and identifies the good elements that arise
and their contribution to the proportions $p_9(\cnqx)$ and $\tilde{p}_9(\cnqsp)$. 
In Sections~\ref{c1}--\ref{c9} we estimate the probability contributions $p_i(\cnqx)$ 
and $\tilde{p}_i(\cnqsp)$ for $i=1,\dots,9$.  Theorem~\ref{main} is proved in Section~\ref{c1},
and Theorems~\ref{broadbrush} and~\ref{main2} are proved in Section~\ref{sec:proof}.

\section{Preliminary results}\label{sec:pre}

We first note that $r$ is a $\ppd(n,q)$ if and only if $q$ has order $n$ modulo $r$. Hence $r \equiv 1\,  (\mbox{mod } n)$ and 
\begin{eqnarray}\label{ppd:ineq}
r=kn+1 
\end{eqnarray}
for some positive integer $k$. In particular, we have $r\geq n+1$.

The following lemma will be used in the computations in Section \ref{c3}.

\begin{lemma}\label{calc:c3}
Let $n,q\geq 2$. Then $f(x)=xq^{2n^2/x}$ is a decreasing function for $1\leq x\leq n$.
\end{lemma}

\begin{proof}
Observe that $f'(x)=q^{2n^2/x}(1-(2n^2\ln q)/x)$. Now $1-(2n^2\ln q)/x<0$ for $1\leq x \leq n$ and the result follows.
\end{proof}

For positive integers $n, q$ with $q\geq 2$,  set $\Theta(0,n;\pm q)=1$, and for an integer $k$ 
satisfying $1\leq k \leq n$ define 

\begin{eqnarray}\label{thet}
\Theta(k,n;q):=\prod_{i=k}^n(1-q^{-i})\quad \mbox{ and }\quad  \Theta(k,n;-q):=\prod_{i=k}^n(1-(-q)^{-i}) .
\end{eqnarray}
Then, by \cite{prastrong}, the orders of the classical groups can be expressed as in 
Table~\ref{order:classical}.
\begin{table}[h]
 \begin{center}
\begin{tabular}{l|ll}\hline
$G$	& $|G|$\\ \hline
$\gl_n(q)$	&$q^{n^2}\Theta(1,n;q)$\\
$\gu_n(q)$	&$q^{n^2}\Theta(1,n;-q)$\\
$\Sp_{2n}(q)$	&$q^{2n^2+n}\Theta(1,n;q^2)$\\
$\Og^\eps_{2n}(q)$,  $\eps=\pm$ &$2q^{2n^2-n}(1-\eps q^{-n})\Theta(1,n-1;q^2)$\\
$\Og_{n}(q)$, $n$ odd     &$2q^{(n^2/2)-(n/2)}\Theta(1,\frac{n-1}{2};q^2)$\\ \hline
\end{tabular}   
 \end{center}
\caption{The orders of the classical groups.}\label{order:classical}
\end{table} 
The following technical lemmas will be useful in our computations.

\begin{lemma}\label{theta:ref}
Assume that $k,n,q$ are integers satisfying $1\leq k\leq n$ and $q\geq 2$. 

\begin{enumerate}
 \item[(i)]   Then
\begin{align*} 
1-\frac{1}{q}-\frac{1}{q^2}& < \Theta(k,n;q) < 1                 &&\\ 
1                          &<\Theta(k,n;-q)  \le 1+\frac{1}{q^k} &&\text{if $k$ is odd}, \\
1-\frac{1}{q^k}            &<\Theta(k,n;-q)  <1                  &&\text{if $k$ is even.}
             \end{align*}
\item[(ii)] Moreover if $k<n$ then
\begin{align*} 
1                          & < \frac{\Theta(k+1,n;q)}{\Theta(1,n-k;q)} < (1-\frac{1}{q}-\frac{1}{q^2})^{-1} &&\\ 
\frac{1-q^{-k-1}}{1+q^{-1}}& < \frac{\Theta(k+1,n;-q)}{\Theta(1,n-k;-q)} < 1 &&\text{if $k$ is odd}, \\
\frac{1}{1+q^{-1}}         & < \frac{\Theta(k+1,n;-q)}{\Theta(1,n-k;-q)} < 1+\frac{1}{q^{k+1}} &&\text{if $k$ is even}, \\
             \end{align*}
\item[(iii)] and if $n$ is even, $n\geq4$ then
\[
\frac{1-q^{-2}-q^{-4}}{1-q^{-n}} < \frac{\Theta(1,\frac{n}{2}-1;q^2)}{\Theta(\frac{n}{2},n-1;q^2)} < \frac{1}{1-q^{-n}}. 
\]

\end{enumerate}
\end{lemma}

\begin{proof}
The inequalities in (ii) and the lower bound for the first line of (ii) are the content 
of \cite[Lemma 3.5]{NP} \cite[Lemmas 3.1 and 3.2]{prastrong}. The upper bound for the first 
line of (ii) follows from two applications of the first line of (i). Similarly the second and third lines of (ii) 
follow from multiple applications of the second and third lines of (i). For part (iii), observe that
$\frac{\Theta(1,n/2-1;q^2)}{\Theta(n/2,n-1;q^2)}=\frac{\Theta(1,n/2;q^2)}{(1-q^{-n})\Theta(n/2,n-1;q^2)}$,
and apply the first line of part (ii).
\end{proof}

\begin{lemma}\label{theta:2}
Assume that $n,m,q$ are integers with $1\leq m<n$. Then
\begin{itemize}
\item[(i)] $\frac{\Theta(1,n;-q)}{\Theta(1,m;-q)} \leq \begin{cases}
                                              1 & \mbox{if $m$ is odd}\\
1+q^{-m-1} & \mbox{if $m$ is even.}\\
                                           \end{cases}$
\item[(ii)] $\frac{\Theta(1,m;-q)}{\Theta(1,n;-q)} \leq \begin{cases}
                                              1 & \mbox{if $m$ is even}\\
\frac{1}{1-q^{-m-1}} & \mbox{if $m$ is odd.}\\
                                           \end{cases}$                                           
\end{itemize}
\end{lemma}

\begin{proof}
(i) Assume first that $n$ is odd. For $m$ odd, we have $\frac{\Theta(1,n;-q)}{\Theta(1,m;-q)} = (1-q^{-m-1})(1+q^{-m-2})\cdots (1+q^{-n})$. 
It is clear that, for any $s\geq 1$, the product $(1-q^{-s})(1+q^{-s-1})\leq 1$. Hence the result follows. 
If $m$ is even, then $\frac{\Theta(1,n;-q)}{\Theta(1,m;-q)} = (1+q^{-m-1})(1-q^{-m-2})\cdots (1+q^{-n})$. Our previous argument shows that $(1-q^{-m-2})\cdots (1+q^{-n})\leq 1$ and the result follows. 

Assume now that $n$ is even. Then $\frac{\Theta(1,n;-q)}{\Theta(1,m;-q)}= \frac{\Theta(1,n-1;-q)}{\Theta(1,m;-q)}(1-q^{-n})$. Hence the result follows from the above computations.

(ii) In this case observe that the product $(1+q^{-s})(1-q^{-s-1})\geq 1$ for any $s\geq 1$, $q\geq 2$. Then the results follow from similar computations in part (i).
\end{proof}

\section{Good elements in classical groups}\label{good:classical}

Let $G$ be a classical group isomorphic to $\SL_{2n}(q)$, $\Su_{2n}(q)$, $\Sp_{2n}(q)$ or $\So_{2n}^\eps(q)$, 
and let $V$ denote the underlying vector space $V=V(2n,q^\delta)$ where $\delta=1$ for $\SL_{2n}(q)$, $\Sp_{2n}(q)$ 
and $\So_{2n}^\eps(q)$, and $\delta=2$ for $\Su_{2n}(q)$. We assume that $n \geq 2$ if $G\cong \SL_{2n}(q)$, 
$\Su_{2n}(q)$ or $\Sp_{2n}(q)$, and $n\geq 4$ if $G\cong \So_{2n}^\eps(q)$.

\begin{lemma}\label{good:cent:order}
Suppose that $t$ is a good element in $G=X_{2n}(q)$ with fixed point subspace $U$. 
\begin{itemize}
   \item[(i)] Then $t$ preserves a decomposition $V=U\oplus W$, such that 
$t$ acts irreducibly on $W$, and if $X\ne\SL$ then $U$ and $W$ are
non-degenerate;

\item[(ii)]  $t$ lies in 
a unique cyclic torus $T$ such that $T$ preserves $V=U\oplus W$, $T$ has fixed point subspace $U$, 
$C_G(t)=C_G(T)$, and $|T|$, $|C_G(t)|$, and the parity of $n$ are as in Table~{\rm\ref{order:cent:good}}.
Also $|N_G(T)|=n|C_G(T)|$ and the number of $G$-conjugates of $t$ contained in $T$ is $n$.   

\item[(iii)] Conversely for $G$ of type $X$, for $n$ as in Table~{\rm\ref{order:cent:good}} and each $m\in\Phi^X(n,q)$, $G$ contains 
$\varphi(m)/n$ conjugacy classes of good elements of order $m$ where $\varphi$ is the Euler phi function. 

\end{itemize}

\end{lemma}

\begin{table}[h]
 \begin{center}
\begin{tabular}{l|l|l|l}\hline
$G$	& parity of $n$ &$|T|$&$|C_G(t)|$\\ \hline
$\SL_{2n}(q)$	&any&$(q^n-1)/(q-1)$&$(q^n-1)|\SL_{n}(q)|$\\
$\Su_{2n}(q)$	& odd &$(q^n+1)/(q+1)$&$(q^n+1)|\Su_{n}(q)|$\\
$\Sp_{2n}(q)$	& even &$q^{n/2}+1$&$(q^{n/2}+1)|\Sp_n(q)|$\\
$\So^\eps_{2n}(q)$,  $\eps=\pm$ & even &$q^{n/2}+1$&$(q^{n/2}+1)|\So_n^{-\eps}(q)|$\\ \hline
\end{tabular}   
 \end{center}
\caption{Tori and centralizers for good elements in classical groups.}\label{order:cent:good}
\end{table}

\begin{proof}
(i)\ The good element $t$ acts on the $n$-dimensional quotient $V/U$ and since $o(t)\in\Phi^X(n,q)$, it follows 
that $t$ is irreducible on $V/U$. By Definition~\ref{def:good}, $t$ is semisimple, and 
so by Maschke's Theorem, there exists a $t$-invariant complement $W$ for $U$. Hence 
$t$ preserves the decomposition $V=U\oplus W$ with $t$ irreducible on $W$. 
Moreover $W$ is uniquely determined by $t$.
Suppose now that $X\ne\SL$. Then, for $u \in U$ and $w \in W$, we have
$(u,w)=(u^t,w^t)=(u,w^t)$, 
so $(u,w^t-w)=0$. Notice that $t-\rm{I}$ is nonsingular on $W$, so $\{w^t-w \mid w \in W\}=W$. 
Therefore $u\in W^\perp$ for each $u\in U$, so $U\subseteq W^\perp$.
Since the dimension of $U^\perp$ is $n$, 
we have $W=U^\perp$, so both $U$ and $W$ are nondegenerate.

(ii) By  \cite[Satz II.7.3]{Hu}, the centraliser of 
$t|_W$ in ${\rm GL}(W)$ is a cyclic torus $\hat T$ of order $q^n-1$, 
and it follows, on identifying ${\rm GL}(W)$ with a subgroup of ${\rm GL}(V)$ fixing $U$ pointwise, 
that $T:=\hat T \cap G$ is the unique cyclic torus containing $t$ such that $T$ preserves $V=U\oplus W$ 
and $T$ has fixed point subspace $U$. Now $N_{\GL(V)}(\hat T)$ leaves both $U$ and $W$ invariant, and it follows from
 \cite[Satz II.7.3]{Hu} that $N_{\GL(V)}(\hat T)=\GL(U)\times N_{\GL(W)}(\hat T)$
with $N_{\GL(W)}(\hat T)=\hat T.n$, and $C_{\GL(V)}(t)=C_{\GL(V)}(\hat T)$. Part (ii) and 
the entries in Table~\ref{order:cent:good} now follow for the case $X=\SL$ on intersecting with $\SL(V)$.

Assume now that $X\ne\SL$, so $U=W^\perp$, and $t|_W$ is an irreducible element of order $o(t)$ in the classical group
$H(W)$ induced on $W$ by the setwise stabilizer of $W$ in $G$. It follows that the parity of $n$ is as in Table~\ref{order:cent:good},
and that $W$ is of minus-type if $X=\So^\eps$ (so in this case $U$ is of type $-\eps$). 
Now $C_G(T)=C_G(t)\leq C:= H(U)\times C_{H(W)}(T)$, where $H(U)$ is the classical group induced on $U$
by the setwise stabilizer of $U$ in $G$, and $C_{H(W)}(T)$  is a torus containing $T$ of order 
$q^n+1, q^{n/2}+1, q^{n/2}+1$ for $X=\Su, \Sp, \So^\eps$, respectively. For $X=\Sp, \So^\eps$, we have 
$C_{H(W)}(T)=T$ of order $q^{n/2}+1$, while $T = C_{H(W)}(T)\cap G$ has index $q+1$ in $C_{H(W)}(T)$ for $X=\Su$. Hence
$|T|$ and $|C_G(t)|$ are as in  Table~\ref{order:cent:good}, for these cases also.

In each case, all tori of $H(W)$ of order $|T|$ are conjugate in $H(W)$, and hence 
all tori of $G$ of order $|T|$ and with an 
$n$-dimensional fixed point space are conjugate in $G$. For all types, $|N_G(T)|=n|C_G(T)|$
(see \cite[Prop. 4.3.6, 4.3.10]{kleidman1990} and the last paragraph of the proof of 
\cite[Lemma 4.3.15]{kleidman1990} for type $\So^\eps$). 
In all cases $N_G(T)$ acts by conjugation on the good elements contained in $T$ 
with orbits of length $n$. Thus $T$ contains $n$ good elements that are $N_G(T)$-conjugate to $t$.
On the other hand, if $t'$ is a $G$-conjugate of $t$ and $t'\in T$ then, by the uniqueness of $T$,
an element of $G$ conjugating $t'$ to $t$ lies in $N_G(T)$. It follows that $T$ contains 
exactly $n$ of the $G$-conjugates of $t$, and these elements form a single orbit under the action of $N_G(T)$.

(iii) Conversely, suppose that $n$ is as in Table~\ref{order:cent:good}. Let $T$ be a cyclic 
torus with an $n$-dimensional fixed point subspace and order as in Table~\ref{order:cent:good}.
If $n>2$ then all elements of $T$ of order divisible by a $\ppd(n,q^\delta)$ prime are good elements, 
if $(n,q^\delta)=(6,2)$ then all elements of $T$ of order $9$ are good elements,  
and if $n=2$ then all elements of $T$ of order $(q+1)/(q-1,2)$ are good elements. 
Let $m\in\Phi^G(n,q)$. Then $T$ contains good elements of order $m$, namely the 
$\varphi(m)$ generators of its unique subgroup of order $m$ where $\varphi$ is the Euler phi function. We have shown that, from 
each $G$-conjugacy class of good elements, exactly $n$ lie in $T$ and form a single 
$N_G(T)$-class. Hence there are exactly $\varphi(m)/n$ conjugacy classes of good elements 
of order $m$ in $G$.
\end{proof}

This result leads to the following estimates for the quantity $|C_G(t)|^2/|G|$, for a good element $t$ in $G$.

\begin{corollary}\label{cor:cent:order}
Let $G=X_{2n}(q)$ and let $t\in G$ be a good element. Then 
\[
 \frac{|C_G(t)|^2}{|G|}\leq\begin{cases}
                            \frac{(q^n-1)^2}{(q-1)q^{2n^2-1}}\leq \frac{2}{q^{2n^2-2n}}   &\mbox{if $X=\SL$}\\
                            \frac{16(q^n+1)^2}{15q^{2n^2}} \leq \frac{27}{20q^{2n^2-2n}}  &\mbox{if $X=\Su$}\\
                            \frac{25}{16} q^{-n^2+n}  &\mbox{if $X=\Sp$}\\
                            \frac{25}{9}q^{-n^2+n}  &\mbox{if $X=\So^\eps$}\\
                           \end{cases}
\]
\end{corollary}

\begin{proof}
We use the data from Tables~\ref{order:classical} and \ref{order:cent:good},
the inequalities from Lemma~\ref{theta:ref}, and that $q\geq2$. First let $X=\SL$. Then  
\begin{eqnarray*}
\frac{|C_G(t)|^2}{|G|} &=&\frac{(q^n-1)^2|\SL_n(q)|^2}{|\SL_{2n}(q)|}
=\frac{(q^n-1)^2q^{2n^2}\Theta(1,n;q)^2}{(1-q^{-1})q^{4n^2}\Theta(1,2n;q)}\\
&<& \frac{(q^n-1)^2\Theta(1,n;q)}{(q-1)q^{2n^2-1}\Theta(n+1,2n;q)} < \frac{(q^n-1)^2}{(q-1)q^{2n^2-1}}< \frac{2}{q^{2n^2-2n}} \\
\end{eqnarray*}
applying Lemma~\ref{theta:ref}(ii).  
Next, for $X=\Su$, we use also that $n\geq3$ and $n$ is odd (from Table~\ref{order:cent:good}).
\begin{eqnarray*}
\frac{|C_G(t)|^2}{|G|} &=&\frac{(q^n+1)^2|\Su_n(q)|^2}{|\Su_{2n}(q)|}
=\frac{(q^n+1)^2 q^{2n^2}\Theta(1,n;-q)^2}{(1+q^{-1})q^{4n^2}\Theta(1,2n;-q)}\\
&=& \frac{(q^{n}+1)^2\Theta(2,n;-q)}{q^{2n^2}\Theta(n+1,2n;-q)} 
<  \frac{(q^{n}+1)^2}{q^{2n^2}(1-q^{-n-1})}\leq \frac{16(q^{n}+1)^2}{15 q^{2n^2}}. \\
\end{eqnarray*}
For $X=\Sp$, we have $n$ even and we use Lemma~\ref{theta:ref}(ii) as for the $\SL$-case.
\begin{eqnarray*}
\frac{|C_G(t)|^2}{|G|} &=&\frac{(q^{n/2}+1)^2|\Sp_n(q)|^2}{|\Sp_{2n}(q)|}
=\frac{(q^{n/2}+1)^2 q^{n^2+n}\Theta(1,n/2;q^2)^2}{q^{2n^2+n}\Theta(1,n;q^2)}\\
&=& \frac{(1+q^{-n/2})^2\Theta(1,n/2;q^2)}{q^{n^2-n}\Theta(n/2+1,n;q^2)} 
<  \frac{(1+q^{-n/2})^2}{q^{n^2-n}}. \\
\end{eqnarray*}
If $n=2$ then a direct computation gives $\frac{|C_G(t)|^2}{|G|}<\frac{6}{5} q^{-n^2+n}$, 
while if $n\geq4$ then the last expression above is at most $\frac{25}{16} q^{-n^2+n}$
Finally, if  $X=\So^\eps$, then $n\geq4$ and $n$ is even. Hence, 
using Lemma~\ref{theta:ref}(iii),
\begin{eqnarray}\label{cor-so}
\frac{|C_G(t)|^2}{|G|} &=&\frac{(q^{n/2}+1)^2|\So^{-\eps}_n(q)|^2}{|\So^\eps_{2n}(q)|} \nonumber\\
&=&\frac{(q^{n/2}+1)^2 q^{n^2-n}(1+\eps q^{-n/2})^2\Theta(1,n/2-1;q^2)^2}{q^{2n^2-n}(1-\eps q^{-n})\Theta(1,n-1;q^2)}\nonumber\\
&=& \frac{(1+q^{-n/2})^2(1+\eps q^{-n/2})^2\Theta(1,n/2-1;q^2)}{q^{n^2-n}(1-\eps q^{-n})\Theta(n/2,n-1;q^2)}\nonumber\\ 
&<&  \frac{(1+q^{-n/2})^2(1+2/(q^{n/2}-1))}{q^{n^2-n}(1-q^{-n})}.
\end{eqnarray}
Using $q\geq2$ and $n\geq4$ this quantity is at most $(\frac{25}{16}.\frac{16}{15}.\frac{5}{3})q^{-n^2+n}
=\frac{25}{9}q^{-n^2+n}$. 
\end{proof}

Next we analyse the irreducible subgrops $\la t,t^g\ra$ mentioned in Remark~\ref{rem:main}~(d).

\begin{lemma}\label{spso:even}
Let $G\cong \Sp_{2n}(q)$ where $q$, $n$ are even. Let $m\in \cnqqsp$ and  $t\in G$ be a good element of order $m$. 
Then for $g\in G$,  either $\langle t, t^g \rangle$ is reducible on $V$, or $\langle t, t^g \rangle \leq \So_{2n}^\pm(q)$.
\end{lemma}

\begin{proof}
Note that $\psp_{2n}(q) \cong \Omega_{2n+1}(q)$. Let $\hat V$ be the associated orthogonal space of dimension 
$2n+1$ for $G$ with nonsingular $1$-dimensional radical $R$ and $V=\hat V/R$. Then $t$ acts irreducibly on a
non-degenerate 
$n$-dimensional subspace $\hat W$ of $\hat V$ and fixes pointwise an $(n+1)$-dimensional subspace $\hat U$. In particular 
$\hat W=[\hat V,t] = \langle vt-v \mid v \in \hat V \rangle$ has dimension $n$. Take an element $g\in G$.
Then similarly $[\hat V, t^g]$ is non-degenerate of dimension $n$. 
Thus $Y:= [\hat V,t]+[\hat V,t^g]$ has dimension at most $2n$ and is left invariant by $\langle t,t^g\rangle$.
If $Y\cap Y^\perp \not\leq R$, then $\langle t,t^g\rangle$ leaves invariant the non-zero totally isotropic
subspace $((Y\cap Y^\perp)+R)/R$ of $V$. If $Y\cap Y^\perp\leq R$ and $Y+R\ne \hat V$, then $(Y+R)/R$ is a
non-degenerate proper subspace of $V$ invariant under $\langle t,t^g\rangle$. Finally if $Y\cap Y^\perp\leq R$ 
and $Y+R= \hat V$, then $Y$ is a hyperplane of $\hat V$ not containing $R$, and the stabilizers in $G$ of such 
hyperplanes are $\So_{2n}^\pm(q)$.  
\end{proof}

We finish this section with a result about subspaces invariant under a good element. 

\begin{lemma}\label{lem:tinv}
Let $t\in\SL_{2n}(q)$ be a good element preserving a decomposition $V=U\oplus W$ as in
Lemma~$\ref{good:cent:order}$, and let $Z$ be an arbitrary $t$-invariant subspace of $V$.
Then either $Z\subseteq U$ or $W\subseteq Z$.
\end{lemma}

\begin{proof}
Let $u+w\in Z\setminus\{0\}$ with $u\in U$ and $w\in W$. 
If for all such vectors the component $w$ is zero then $Z\subseteq U$.
Suppose on the other hand that $w\ne0$ for some vector in $Z$.
Then $Z$ contains $(u+w)t^j=u + wt^j$, and hence also $wt^j-w$, for each $j=1,\dots,n-1$.  
Since $t|_W$ is irreducible, the vectors $w,wt,\dots,wt^{n-1}$ are linearly independent, and
hence $wt-w,\dots,wt^{n-1}-w$ generate an $(n-1)$-dimensional subspace of $Z\cap W$.
Since $n\geq2$, $Z\cap W$ is $t$-invariant, and  $t|_W$ is irreducible, it follows that
$W\subseteq Z$.  
\end{proof}

\section{Alternating and symmetric groups}\label{ansn}

In this section we identify good elements that act on the deleted permutation module 
of an alternating or symmetric group (Lemma~\ref{sn:good:type}). We use the following 
notation and definitons. Let $V=\mathbb{F}_q^\ell$, $q=p^k$, and let $M=S_\ell$ act on $V$ by permuting the coordinates naturally. Define 
\[
W:=\{ (x_1, \ldots, x_\ell)\in V \mid \sum x_i=0\} \mbox{ and }  E:=\langle e\rangle, \mbox{ where }
e=(1,1,\ldots, 1).
\]
Then $W$ and $E$ are the only non-zero proper $S_\ell$-invariant subspaces of $V$. Moreover,  
$E$ is contained in $W$ if and only if $p$ divides $\ell$. Defining $U:=W/(W\cap E)$, we have
\[
\dim(U) = \left\{ \begin{array}{cl}
\ell-1 & \mbox{if } p \mbox{ does not divide }\ell\\
\ell-2& \mbox{if } p\mid \ell.
\end{array} \right.
\]
The quotient space $U$ is called the \textit{fully deleted permutation module}.

Let $\mathcal{B}=\{e_1, \ldots, e_\ell \}$ be the standard basis for $V$. Then $v_i=e_i-e_{i+1} + W\cap E$, for 
$1\leq i\leq m$, forms a basis $\mathcal{B}'=\{v_1, \ldots , v_m\}$ for $U$, where $m=\ell -2$ or $\ell -1$ 
depending on whether $p$ divides $\ell$, or $p$ does not divide $\ell$, respectively.
If $p$ is an odd prime then for $n\geq3$, by \cite[Section 4]{liebeck85.426},  
\begin{itemize}
\item[[1.]] $S_{2n+1}$ is contained in ${\rm PO}^\eps_{2n}(p)$  if $p \nmid 2n+1$;
\item[[2.]] $S_{2n+2}$ is contained in ${\rm PO}^\eps_{2n}(p)$ if $p \mid n+1$.
\end{itemize}
If $p=2$, then
\begin{itemize}
\item[[3.]] $S_{2n+2}$ is contained in  $ \left\{ \begin{array}{ll}
\Og^+_{2n}(2) & \mbox{ if }n\equiv 3\, (\mbox{mod 4}),\\
\Og^-_{2n}(2) & \mbox{ if }n\equiv 1\, (\mbox{mod 4}),\\
\Sp_{2n}(2) & \mbox{ if }n \mbox{ is even;}
 \end{array} \right.$
\end{itemize}

\begin{itemize}
\item[[4.]] $S_{2n+1}$ is contained in  $ \left\{ \begin{array}{ll}
\Og^+_{2n}(2) & \mbox{ if }n\equiv 0\, (\mbox{mod 4}),\\
\Og^-_{2n}(2) & \mbox{ if }n\equiv 2\, (\mbox{mod 4}).
 \end{array} \right.$
\end{itemize}

We note that, if $X_{2n}(p)$ is one of the orthogonal groups in [1]--[4] above, then 
$S_{2n+1}\times Z$ or $S_{2n+2}\times Z$, as appropriate, is 
a maximal $\cc_9$-subgroup where $Z$ is the centre of $X_{2n}(p)$ and $|Z|=(p-1,2)$.
We use this notation throughout this section to examine good elements in these groups.

\begin{lemma}\label{sn:forc2}
Suppose that $S_\ell\times Z$ lies in an orthogonal group $X_{2n}(p)$ as above, where
$n\geq 3$ and $\ell \in \{ 2n+1,2n+2 \}$.  
If $S_\ell\times Z$ contains a good element $t$, then the largest prime $r$ dividing 
$o(t)$ is $r=n+1$ or $r=2n+1$, and $r$ is a $\ppd(n;p)$ prime. 
\end{lemma}

\begin{proof}
By Definition~\ref{def:good}, the order of $t$ is divisible by a $\ppd(n;p)$ prime, say $r$, and by (\ref{ppd:ineq}), 
$r=kn+1$ for some $k \ge 1$. Since $\ell < 3n+1$, the only possibilities are $k=1$ or $k=2$.
If $s>r$ and $s$ is a prime dividing $o(t)$, then $s+r\leq \ell\leq 2n+2\leq 2r$, which is a contradiction.
\end{proof}

Recall that  $\mathcal{B}=\{e_1, \ldots, e_\ell \}$ is the standard basis for $V=\mathbb{F}_p^\ell$. 
Let $k_1,\ldots,k_s$ be positive integers with $\sum_{i=1}^s k_i = \ell$ and, for 
$1\le i \le s$, let $\ell_i:=\sum_{j<i} k_j$ (so $\ell_1=0$). For an integer $j$, let $\delta_{j} = 0$ if $j$ is even and $\delta_{j} = 1$ if $j$ is odd.
For $1 \le i \le s$ and 
$a_i, b \in \mathrm{F}_p$, define

\begin{align}\label{vectorv1}
v(i,a_i,1,b)&= \sum_{j=1}^{k_i} (a_i+(j-1)b)e_{\ell_i+j},\nonumber\\
v(i,a_i,-1,b)&= \sum_{j=1}^{k_i} ((-1)^{j-1}a_i+\delta_{j-1}b)e_{\ell_i+j}.
\end{align}

\begin{lemma}\label{sn:good:fix}
Using the notation $V, W, E=\langle e\rangle, e_1,\ldots,e_\ell$ as above, let $a=\pm1$, $g \in S_\ell$ and 
let $g=w_1 \cdots w_s$ be the disjoint cycle decomposition of $g$, with $w_i=(\ell_i+1,\ldots, \ell_i+k_i)$. 
Let $v\in V$. Then $v+ W\cap E$ is fixed by $ag$ if and only if there exist $a_1,\dots,a_s,b\in\mathbb{F}_q$ 
such that  $v=\sum_{i=1}^s v(i,a_i,a,b)$, with $b=0$ if $W\cap E=0$, and for each $i\leq s$,
\begin{equation}\label{agfix}
(1+a+\dots+a^{k_i-1})((a-1)a_i+b)=0.   
\end{equation}
Moreover the vector  $v=\sum_{i=1}^s v(i,a_i,a,b)$ lies in $W$ if and only if 
\begin{align}\label{sumvi}
\sum_{i=1}^s k_ia_i + \left(\sum_{i=1}^s\binom{k_i}{2}\right)b=0 \qquad&\mbox{if $a=1$}\nonumber\\
  \sum_{i=1}^s \delta_{k_i}a_i + \left(\sum_{i=1}^s\lfloor\frac{k_i}{2}\rfloor\right)b=0 \qquad &\mbox{if $a=-1$.} 
\end{align}
Also both conditions, `$v\in W$' and `$v+ W\cap E$ is fixed by $ag$', hold if and only if the elements $a_i, b$ and the numbers
$I_1=\#\{ i\mid p\nmid k_i\}$ and $I_2=\#\{ i\mid k_i\equiv 2\pmod{4}\}$
 are as in one of the rows of Table~$\ref{tbl:agfix}$. 

\begin{center}
\begin{table}
\begin{tabular}{|c|c|c|c|}
\hline 
$a$ & $p$ & Conditions on the $k_i$ & Conditions on the $a_i$ and $b$ \\ 
\hline 
1 & odd & $I_1=0$ &  $b=0$ if $p\nmid \ell$ \\ 
\hline 
• & $2$ & $I_1=0$, $I_2$ even & $b=0$ if $p\nmid \ell$ \\ 
\hline 
• & $2$ & $I_1=0$, $I_2$ odd & $b=0$ \\ 
\hline 
• & -- & $I_1>0$ & $b=0$ and  $\sum_{i=1}^sk_ia_i=0$ \\ 
\hline 
$-1$ & odd & -- &  $a_i=b/2$ for each odd $k_i$, \\
  &     &        &  and $b=0$ if $p\nmid \ell$\\ 
\hline 
\end{tabular}
\caption{Conditions for $v\in W$ to be fixed by $ag$ in Lemma~\ref{sn:good:fix}.} \label{tbl:agfix}
\end{table}
\end{center}

\end{lemma}

\begin{proof}
The coset $v+ W\cap E$ is fixed by $ag$ if and only if $avg = v - be$ for some $b\in\mathbb{F}_q$,
with $b=0$ if $W\cap E=0$. 
In order for this equation to hold, for each $i$ the coefficient $a_i$ of $e_{\ell_i+1}$ in $v$ determines the 
coefficients
of $e_{\ell_i+j}$ for $2\leq j\leq k_i$ recursively (equating coefficients of these basis vectors on both sides of the 
equation): namely, 
the coefficient of  $e_{\ell_i+2} = aa_i+b,\dots,$ the coefficient of  $e_{\ell_i+k_i}=a^{k_i-1}a_i + 
(1+a+\dots+a^{k_i-2})b$, that is to say, $e_{\ell_i+j}=a_i + (j-1)b$ if $a=1$ and $(-1)^{j-1}a_i+\delta_{j-1}b$ if $a=-1$. 
Thus the components of $v$ over the cycle $w_i$ add to $v(i,a_i,a,b)$.  Moreover, equating the coefficients 
of $e_{\ell_i+1}$ on both sides of the equation yields $(a^{k_i}-1)a_i + (1+a+\dots+a^{k_i-1})b=0$, or equivalently, $(1+a+\dots+a^{k_i-1})((a-1)a_i+b)=0$. Conversely, if 
this equation holds for each $i$ and $v=\sum_{i=1}^sv(i,a_i,a,b)$, then it is straightforward to check that
$avg = v - be$, and hence $ag$ fixes $v+ W\cap E$. This proves the first assertion.

It is easy to see that the coefficients of the standard basis elements in $v=\sum_{i=1}^sv(i,a_i,a,b)$ 
add up to $0$ (that is to say, $v\in W$) if and only if the equations \eqref{sumvi} hold.

Next we examine more carefully the combined conditions:  `$v\in W$'
and `$v+ W\cap E$ is fixed by $ag$', for $v=\sum_{i=1}^sv(i,a_i,a,b)$. By the arguments above, these 
two conditions hold if and only if both \eqref{agfix} holds for each $i$ and also the equations 
in \eqref{sumvi} hold. 
Note also that $b=0$ if $W\cap E=0$, equivalently, if $p\nmid \ell$

Suppose that $a=1$. Assume first that $p$ divides $k_i$ for each $i$, that is $I_1=0$. Then  
\eqref{agfix} holds for each $i$ and the final condition becomes $(\sum_{i=1}^s\binom{k_i}{2})b=0$. 
If $p$ is odd then this holds since $p$ divides each $\binom{k_i}{2}$, and row 1 of Table~\ref{tbl:agfix} holds. 
Similarly if $p=2$ and the number of $i$ such that $k_i\equiv 2\pmod{4}$ is even, that is, $I_2$ 
is even, then the final condition holds, and we have row 2 of Table~\ref{tbl:agfix}. 
On the other hand, if $p=2$ and $I_2$ is odd, then the final condition holds if and only 
if $b=0$, as in row 3 of Table~\ref{tbl:agfix}.
Now assume that $I_1>0$, say $p\nmid k_{i_1}$. Then \eqref{agfix} holds for $i_1$ if and only if  
$b=0$. With $b=0$ we see that \eqref{agfix} holds for each $i$, and the final condition becomes 
$\sum_{i=1}^sk_ia_i=0$, as in row 4 of Table~\ref{tbl:agfix}.

Now consider the case $a=-1\ne 1$, so $p$ is odd. If $k_i$ is even
then the first factor on the left hand side of \eqref{agfix} is zero, and hence 
\eqref{agfix} holds. On the other hand if $k_{i}$ is odd, Then  \eqref{agfix} holds if and only 
if $a_{i}=b/2$. Thus \eqref{agfix} holds for all $i$ if and only if $a_i=b/2$ whenever $k_i$ is odd,
Under these conditions,  \eqref{sumvi} becomes 
\[
0=\sum_{\text{$i$ with $k_i$ odd}} a_i +\left(\sum_{i=1}^s\lfloor\frac{k_i}{2}\rfloor\right)b 
= 2^{-1} \left(\sum_{i=1}^s k_i\right) b = 2^{-1} \ell b
\]
which, since $p$ is odd, is equivalent to `$b=0$ if $p\nmid\ell$', as in row 5 of Table~\ref{tbl:agfix}.
\end{proof}

Now we are ready to compute the dimensions of fixed point spaces for elements of $Z\times S_\ell$ in the fully deleted permutation module $U$. 

\begin{lemma}\label{sn:cg}
Suppose that $p \mid 2n+2$. Let $g \in S_{2n+2}$ and let $g=w_1 \cdots w_s$ be the disjoint 
cycle decomposition of $g$, with $w_i=(\ell_i+1,\ldots, \ell_i+k_i)$, and let  $I_3=
\#\{ i\mid k_i \mbox{ odd }\}$. Then $\dim(\mathrm{Fix}_{U}(\pm g))$ is given in Table~$\ref{tbl:cg}$. 
\end{lemma}

\begin{center}
\begin{table}
\begin{tabular}{|c|c|c|c|}
\hline 
$a$ & $p$ & $\dim(\mathrm{Fix}_{U}(ag))$ & Conditions on the $k_i$ \\ 
\hline 
1 & odd & $s  $ &  $p\mid k_{i}$ for all $i$ \\ 
1 & odd & $s-2$ &  $p\nmid k_{i}$ for some $i$ \\ 
-1& odd & $s-I_3$ &    \\ 
\hline 
1 & $2$ & $s  $ & all $k_i$ even, and $n$ odd \\ 
1 & $2$ & $s-1$ & all $k_i$ even, and $n$ even \\ 
1 & $2$ & $s-2$ & some $k_i$ odd \\ 
\hline 
\end{tabular}
\caption{Fixed point dimensions for $\pm g$ in Lemma~\ref{sn:cg}.} \label{tbl:cg}
\end{table}
\end{center}

\begin{proof} Note that $E<W$ since $p$ divides $\ell=2n+2$, so $U=W/E$. Let $g$ be as given and $a=\pm1$,
let $I_1, I_2$ be as in Lemma~\ref{sn:good:fix}, and $I_3$ be as in the statement.  
By Lemma~\ref{sn:good:fix} the fixed vectors of $ag$ in $U$ are the cosets $v+E$ with 
$v=\sum_{i=1}^sv(i,a_i,a,b)$ where the $a_i$ and $b$ satisfy the appropriate line of Table~\ref{tbl:agfix}.
Note that since $ag$ fixes $E$, either all no vectors in a coset are of this form.  
Also note the assumption that $p$ divides $\ell=2n+2$.
Suppose first that $p$ is odd and $a=1$.  
If $p\mid k_i$ for all $i$ (line 1 of Table~\ref{tbl:agfix}), then the $a_i$ and $b$ are arbitrary 
giving $p^{s+1}$ vectors forming $p^s$ cosets, so $\dim(\mathrm{Fix}_{U}(g))=s$, as in line 1 of Table~\ref{tbl:cg}.    
On the other hand if $p\nmid k_i$ for some $i$ (line 4 of Table~\ref{tbl:agfix}), then we require $b=0$ and 
$\sum_{i\in I_1}k_ia_i=0$, giving $p^{s-2}$ cosets and dimension $s-2$ as in line 2 of Table~\ref{tbl:cg}.
If $p$ is odd and $a=-1$ (line 5 of Table~\ref{tbl:agfix}), then $b$ is arbitrary since $p\mid\ell$, and the requirements are that 
$a_i=b/2$ for  each odd $k_i$, giving $p^{s-I_3}$ cosets and dimension $s-I_3$, as in line 3 of  Table~\ref{tbl:cg}.

Now consider $p=2$ so $a=1$. If all the $k_i$ are even (lines 2 and 3 of Table~\ref{tbl:agfix}), 
we obtain dimension $s$ or $s-1$ according as $I_2$ is even or odd respectively, and this is equivalent to 
$\ell=2n+2\equiv 0$ or $2\pmod{4}$, respectively, that is, 
$n$ being odd or even, respectively, as in line 4 or 5 of  Table~\ref{tbl:cg}, respectively. 
\end{proof}

\begin{lemma}\label{sn:cg:2}
Suppose that $p \nmid 2n+1$. Let $g \in S_{2n+1}$ and let $g=w_1 \cdots w_s$ be the disjoint 
cycle decomposition of $g$, with $w_i=(\ell_i+1,\ldots, \ell_i+k_i)$. Then 
$\dim(\mathrm{Fix}_{U}(g))=s-1$, and for $p$ odd, $\dim(\mathrm{Fix}_{U}(-g))=s-I_3$ with 
$I_3=
\#\{ i\mid k_i \mbox{ odd }\}$.
\end{lemma}

\begin{proof}
Since $p \nmid 2n+1$, $W\cap E=0$ and so $U=W/(W\cap E)\cong W$. Also  $p \nmid 2n+1$
implies that not all cycle lengths in $g$ are divisible by $p$, so by Lemma~\ref{sn:good:fix}, $I_1>0$ 
and the fixed vectors of $g$ in $W$ are the sums $\sum_{i=1}^sv(i,a_i,1,0)$ with $\sum_{i\in I_1} k_ia_i=0$,
giving fixed point dimension $s-1$. For $p$ odd, the fixed vectors of $-g$ are the sums 
 $\sum_{i=1}^sv(i,a_i,-1,0)$ with $a_i=0$ for each odd $k_i$, giving fixed point dimension $s-I_3$.
\end{proof}

\begin{lemma}\label{sn:good:type}
Let $G\cong \So_{2n}^\pm(p)$ or $\Sp_{2n}(2)$ and suppose that $M\cong Z\times S_\ell <G$,  
where $\ell=2n+1$ or $2n+2$, $\ell\geq5$, and $|Z|=(2,q-1)$,
and that $M$ contains a good element. Then $n\geq4, n+1$ is prime, and the good elements in $M$ 
lie in $S_\ell$ and have 
cycle structure $(n+1)^11^{\ell-n-1}$.
\end{lemma}

\begin{proof}
Recall that by Lemma \ref{good:cent:order}, $G$ contains good elements if and only if $n$ is even. 
If $n=2$ then $G=\So_4^-(2)\cong A_5$ or  $G=\Sp_4(2) \cong A_6$ so $M$ is not a proper subgroup of $G$. 
Hence $n \ge 4$. Let $t = ag\in M$ be a good element, with $a\in\mathbb{F}_p^*$, $g\in S_\ell$. 
As $|Z|=(p-1,2)$, $a=\pm1$. Let $c(g)$ 
denote the number of cycles in the disjoint cycle decomposition of $g$. By Lemma~\ref{sn:forc2}, $g$ 
has a cycle of prime length $r\geq n+1$, implying 
that $c(g) \le n+ \delta$, with $\delta:=\ell-2n\in\{1,2\}$. Note that $c(g)=n+ \delta$ if and only if $r=n+1$ 
and the cycle structure of $g$ is $(n+1)^11^{n+\delta-1}$, that is, the asserted cycle structure.

By Lemmas~\ref{sn:cg} and~\ref{sn:cg:2}, $n=\dim(\mathrm{Fix}_{U}(ag))\leq c(g)$.
If $g$ has no cycles of length 1 then $c(g) \le 1+\lfloor (n+1/2) \rfloor <n$, which is a contradiction. Hence 
$g$ has at least one cycle of length 1, and then the same lemmas imply that 
$\dim(\mathrm{Fix}_{U}(ag))\leq c(g)-\delta$ (note that in line 3 of Table~\ref{tbl:cg} we have $I_3\geq2$ since
there are cycles of odd lengths $1, r$). Thus $n=\dim(\mathrm{Fix}_{U}(ag))\leq c(g)-\delta\leq n$,
and equality must hold. By the previous paragraph, $g$ has the required cycle structure and so $r=n+1$ is prime.

If $a=1$ there is nothing more to prove since in this case $t=g\in S_\ell$. 
Assume then that $p$ is odd and $a=-1$, so $t=-g$. We have proved that 
$\dim(\mathrm{Fix}_{U}(-g))= c(g)-\delta=n$. However line 3 of Table~\ref{tbl:cg} gives  
$\dim(\mathrm{Fix}_{U}(-g))=c(g)-I_3$, where $I_3$ is the number of odd $k_i$.
Hence $I_3=\delta\leq 2$. However, since $g$ has cycle structure $(n+1)^11^{n+\delta-1}$,
the parameter $I_3=n+\delta>2$, which is a contradiction.
\end{proof}

Using this information we obtain an estimate for the contribution of these subgroups to the 
probabilities $p_9(\cnqx)$ and $\tilde{p}_9(\cnqsp)$ (which is reasonable only for $n\geq5$).

\begin{lemma}\label{lem:c9-altprobs}
Let $p^{(i)}_9(\cnqx)$ (for $X=\Sp$ or $\So^\eps$) and $\tilde{p}^{(i)}_9(\cnqsp)$ 
denote the contributions to $p_9(\cnqx)$ and $\tilde{p}_9(\cnqsp)$ from maximal subgroups
$Z\times S_\ell$ or $Z\times A_\ell$,  where $\ell=2n+1$ or $2n+2$, $\ell\geq7$, and $|Z|=(2,q-1)$. 
Then  $p^{(i)}_9(\cnqx)<q^{-n^2+4n+3}$ and
 $\tilde{p}^{(i)}_9(\cnqsp)<q^{-n^2+4n+3}$. 
\end{lemma}

\begin{proof}
Suppose that $G=X_{2n}(q)$ contains a maximal subgroup $M\cong Z\times S_\ell$ or $Z\times A_\ell$, with 
$\ell = 2n+1$ or $2n+2$ and $X=\So^\eps$ or $\Sp$, as in [1.]--[4.] above.

We claim that the number of $G$-conjugacy classes of such subgroups is 1, unless,
possibly, $\ell=6$ or $9$.
By \cite[Lemma 1.7.1]{kleidman87.173}, the number of $G$-conjugacy 
classes of such subgroups $M$ is at most the number of $L$-conjugacy classes, where $L:=\gl_{2n}(q)$.
Apart from the exceptions $\ell=6, 9$, there is a unique irreducible representation of $A_\ell$ of degree 
$2n$ in characteristic $p$ (see for example 
\cite[Proposition 5.3.5]{kleidman1990} for $\ell\geq10$ and the character tables in 
\cite{atlas, braueratlas} for smaller $\ell$).  Thus
for $M, N$ two maximal subgroups of $G$ of this type, the derived subgroups $M', N'$ (both isomorphic 
to $A_\ell$) are conjugate in $L$. Hence 
there is a single $G$-conjugacy class of subgroups $M'\cong A_\ell$, and a single $G$-conjugacy class of
subgroups $M=N_G(M')$. There are $2, 3$ irreducible representations of $A_6, A_9$ of degrees 
$2n=4, 8$, respectively, in certain characteristics, but these values for $n$ are excluded.

By Lemma \ref{sn:good:type}, the good elements $t\in M\cap \cal C$ all lie in $S_\ell$ and have cycle type 
$(n+1)^11^{\ell-n-1}$. All such elements are conjugate in $M$. 
Thus by \eqref{prob:cc}, \eqref{main3:prob} and \eqref{c1:singleclass}, the contributions  to $p_9(\cnqx)$ are 
\[
   p_9^{(i)}(\cnqx) = \frac{|C_G(t)|^2}{|G|} \frac{|M|}{|C_M(t)|^2} = \frac{|C_G(t)|^2}{|G|} \frac{\ell! (2,q-1)}{(n+1)^2(\ell-n-1)!^2},
\]
and similarly for $\tilde{p}^{(Sym)}_9(\cnqsp)$.
By Corollary~\ref{cor:cent:order}, $\frac{|C_G(t)|^2}{|G|}\leq \frac{25}{9}q^{-n^2+n}$. Also the second factor is
$\frac{\binom{2n+2}{n+1} (2,q-1)}{(n+1)^2}$ or $\frac{\binom{2n+1}{n} (2,q-1)}{n+1}$ as $\ell=2n+2$ or $2n+1$ 
respectively, and the former is greater. Observing that  ${2n+2 \choose n+1} \leq (2e)^{n+1}<q^{3n+3}$, we see that
\[
   p_9^{(i)}(\cnqx) \leq \frac{50}{9q^{n^2-n}}  \frac{(2e)^{n+1}}{(n+1)^2} < \frac{1}{q^{n^2-4n-3}}.
\]
The same argument shows that 
 $\tilde{p}^{(i)}_9(\cnqsp)<1/q^{n^2-4n-3}$.
\end{proof}

\section{{$\mathbf{C}_1$:} Subspace stabilizers and Theorem~\ref{main}}
\label{c1}

Here $G=X_{2n}(q)$, and we consider subgroups $M$ that are maximal subject to being reducible on $V$. 

If $G=\SL_{2n}(q)$, then each such subgroup $M\cong P_i$, the stabilizer of some $i$-dimensional subspace $Z$,
and has the form $M=Q_ML_M$ where $Q_M$ is the unipotent radical of $M$ and $L_M=(\gl_i(q) \times \gl_{2n-i}(q))\cap G$
is the stabilizer of some decomposition $V=Z\oplus Y$ and is called a {\em Levi factor} of $M$. Also $M$
is isomorphic under an outer automorphism of 
$G$ to $P_{2n-i}$.  

For $G=X_{2n}(q)$ with $X\ne \SL$, a maximal reducible subgroup $M$ preserving a proper nontrivial 
subspace $Z$ also leaves invariant $Z\cap Z^\perp$. By maximality, $Z\cap Z^\perp$ is either $0$ or $Z$, that is, 
$Z$ is either nondegenerate, or totally isotropic. In the former case, $M$ is the 
direct product of classical groups induced on $Z$ and $Z^\perp$, and we may assume  
that $i=\dim(Z)\leq n$ since $M=N_i$ is also the stabilizer of $Z^\perp$ of dimension $2n-i$.  
We note that, if $i=n$ then  $M$ may not be maximal, since 
there may exist elements of $G$ interchanging $Z$ and $Z^\perp$.
In the latter case, $M$ is a maximal parabolic subgroup $P_i$, 
the stabilizer of a totally isotropic $i$-dimensional subspace 
$Z$ of $V$ (where $i\leq n$), and $M$ has the form $M=Q_ML_M$, where 
$Q_M$ is the unipotent radical and $L_M$ is $\gl_i(q^\delta) \times G_{2n-2i}(q)$, a {\em Levi factor}  of $M$,  
with $G_{2n-2i}(q)$ a classical group induced by $L_M$ on $Z^\perp/Z$, \cite[Lemma 4.1.12]{kleidman1990}.

Often, if $t, t^g$ are good elements generating a reducible subgroup, then they fix more than one subspace.
For our estimates, we wish to avoid `double counting' such pairs as far as possible.
Thus first we identify, in Lemmas~\ref{c1cases} and~\ref{c1casessl}, the maximal reducible subgroups we need to consider.
Throughout the section, $\mathcal C$ denotes the $G$-conjugacy class of good elements containing $t$.

\begin{lemma}\label{c1casessl}
Let $t, t^g$ be good elements in $\SL_{2n}(q)$ such that $\la t,t^g\ra$ 
is reducible on $V$, and let $V=U\oplus W$ be the decomposition preserved 
by $t$ as in Lemma~$\ref{good:cent:order}$. Then 
$\la t,t^g\ra\leq M$  for one of the following subgroups $M$.
\begin{description}
 \item[(a)] $M=P_{2n-1}$ preserving a $(2n-1)$-space containing $W+W^g$; or
\item[(b)]  $M\cong P_1$ preserving a $1$-space contained in $U\cap U^g$; or 
\item[(c)] $M=P_n$ preserving either $U=W^g$ or $W=U^g$. 
\end{description}
Moreover, for $M$ in parts (a) and (b), $M\cap\mathcal C$ is a single 
$M$-conjugacy class, while in part (c), $M\cap\mathcal C$ is a union of two $M$-conjugacy classes and $t,t^g$ are not $M$-conjugate.
\end{lemma}

\begin{proof}
Let $Z$ be a proper nontrivial subspace preserved by $\la t,t^g\ra$,
and note that $t^g$ preserves $V=U^g\oplus W^g$. Since $t$ leaves $Z$ invariant, it 
follows from Lemma~\ref{lem:tinv} that either $Z\subseteq U$ or $W\subseteq Z$, and similarly 
either $Z\subseteq U^g$ or $W^g\subseteq Z$.
Suppose first that both $W\subseteq Z$ and $W^g\subseteq Z$.  
Since $V=U\oplus W$, each coset of $Z$ is of the form $Z+u$ for some $u\in U$. 
Since $t$ fixes $u$, $t$ also fixes the coset $Z+u$. Similarly $Z+u=Z+(u')^g$ for some
$u'\in U$ and is therefore fixed by $t^g$. It follows that $\la t,t^g\ra$ fixes each coset 
of $Z$ setwise, and hence $\la t,t^g\ra$ fixes a $(2n-1)$-dimensional 
subspace $Z'$ containing $Z$. Thus $\la t,t^g\ra$ lies in $M\cong P_{2n-1}$.

Suppose next that $Z\subseteq U$ and $W^g\subseteq Z$. Then $Z=U=W^g$ and  
$\la t,t^g\ra$ lies in $M\cong P_{n}$. Similarly if $Z\subseteq U^g$ and 
$W\subseteq Z$, then $Z=U^g=W$ and again  
$\la t,t^g\ra$ lies in $M\cong P_{n}$. 
From now on we may assume that $Z\subseteq U\cap U^g$, and
hence that  Since $Z\ne 0$
the group $\la t,t^g\ra$ fixes $Z$ pointwise and hence lies in $M\cong P_{1}$.

In part (a) or (b), $M$ contains the cyclic torus $T$ containing $t$ (defined
in Lemma~\ref{good:cent:order}), all such cyclic tori contained in $M$ are 
conjugate in $M$, and $|N_M(T)| = n|C_M(T)|$, so that the $n$ elements of $T
\cap\mathcal C$ are still conjugate in $N_M(T)$. Thus $M \cap\mathcal C$ is 
a single $M$-conjugacy class. On the other hand in part
(c) there are two $M$-conjugacy classes of tori which are $G$-conjugate to $T$, but
we still have that $|N_M(T)| = n|C_M(T)|$, so the $n$ elements of $T \cap\mathcal C$ 
are still conjugate in $N_M(T)$. Thus $M \cap\mathcal C$ is a union of two 
$M$-conjugacy classes, and since exactly one of $t, t^g$ fixes $Z$ pointwise, 
these elements lie in different $M$-classes in
$M \cap\mathcal C$.
\end{proof}

\begin{lemma}\label{c1cases}
Let $t, t^g$ be good elements in $X_{2n}(q)$, where $X\ne\SL$, such that $\la t,t^g\ra$ 
is reducible on $V$, and let $V=U\oplus W$ be the decomposition preserved 
by $t$ as in Lemma~$\ref{good:cent:order}$. Then 
$\la t,t^g\ra\leq M$  for one of the following subgroups $M$.
\begin{description}
   \item[(a)] $M\cong P_1$, preserving a $1$-space contained in $U\cap U^g$; or
\item[(b)]    $M\cong N_1$ preserving a non-degenerate $1$-space contained in $U\cap U^g$, and
$X=\Su$ or $\So^\eps$; or 
\item[(c)]   $M\cong N_n$, preserving the non-degenerate subspaces $U=W^g$ and $W=U^g$, and 
$X=\Su, \Sp, \So^+$, where if $X=\So^+$ then $U, W$ have minus type.  
\end{description}
Moreover, for $M$ in parts (a) and (b), $M\cap\mathcal C$ is a single 
$M$-conjugacy class, while in part (c), $M\cap\mathcal C$ is a union of two $M$-conjugacy classes and $t,t^g$ are not $M$-conjugate.
\end{lemma}

\begin{proof}
Let  $Z$ be a proper nontrivial subspace preserved by $\la t,t^g\ra$. 
Since also $Z\cap Z^\perp$ is preserved we may assume that $Z$ is either
totally isotropic or nondegenerate. 

Suppose first that $Z$ is totally isotropic
of dimension $i\leq n$, and let $M=Q_ML_M\cong P_i$ be the maximal parabolic subgroup stabilizing $Z$, as above. 
We claim that $i\leq n/2$.  Let $m=o(t)$. If $i=n$, then $L_M=\gl_{n}(q^\delta)$ 
does contain elements of order $m$ but each such element 
fixes no non-zero vector of $V$ and hence is not a good element. Thus $P_n$ contains no good elements, so  $i<n$. 
If $n/2<i<n$, then neither $\gl_{i}(q^\delta)$ nor $G_{2n-2i}(q)$ contains elements of order $m$,
and hence in this case also $P_i$ does not contain any good elements. This proves the claim.
Since $i\leq n/2$ it follows from Lemma~\ref{lem:tinv} that $Z\subseteq U\cap U^g$ is fixed pointwise by 
$\la t,t^g\ra$. Thus $\la t,t^g\ra$ is contained in a parabolic subgroup $P_1$ stabilising 
a $1$-subspace of $U\cap U^g$ as in part (a). 

Now suppose that $Z$ is nondegenerate. Since $\la t,t^g\ra$ also preserves $Z^\perp$ 
we may assume that $i:=\dim(Z)\leq n$. Then it follows from Lemma~\ref{lem:tinv} that 
either (i) $Z$  is contained in $U\cap U^g$ and hence is fixed 
pointwise by $\la t,t^g\ra$ (possibly interchanging $Z$ and $Z^\perp$ if $i=n$), or 
(ii) $i=n$ and, interchanging $Z$ and $Z^\perp$ if necessary,  
$Z=U=W^g$ and $Z^\perp=W=U^g$.

Consider case (i). If $Z$ contains an isotropic vector  $v$, then $\langle t,t^g\rangle$ is 
contained in the stabilizer of $\langle v\rangle$, a parabolic subgroup $P_1$, as in part (a).
If $Z$ contains no isotropic vectors, then either  $i=1$ and $X=\Su$ or $\So^\eps$, or
$i=2$, $X=\So^\eps$, and $Z$ is of minus type. In either case $\langle t,t^g\rangle$
preserves a $1$-subspace of $Z$ and part (b) holds. 

Finally suppose that $i=n$, $Z=U=W^g$ and $Z^\perp=W=U^g$.
Here, if $X=\So^\eps$, then $U, W$ must both have minus type (to admit the actions 
of $t, t^g$) and so we have $\eps=+$. 

In parts (a) or (b), $M$ is the stabilizer of a $1$-space $Z$ and in both cases
$W + W^g \subseteq Z^\perp$. It follows from the structure of $M$ that the cyclic tori $T, T^g$
containing $t, t^g$ respectively (as defined in Lemma~\ref{good:cent:order}) 
are conjugate in $M$. Also
$|N_M(T)| = n|C_M(T)|$, and hence $M \cap\mathcal C$ is a single $M$-conjugacy class. On
the other hand, in part (c), there are two $M$-conjugacy classes of tori which
are $G$-conjugate to $T$, but we still have that $|N_M(T)| = n|C_M(T)|$, so the $n$
elements of $T \cap\mathcal C$ are still conjugate in $N_M(T)$. Thus $M \cap\mathcal C$ 
is a union of two $M$-conjugacy classes, and since exactly one of $t, t^g$ fixes $U$ 
pointwise, and $U$ is $M$-invariant, these elements lie in different $M$-classes in $M \cap\mathcal C$.
\end{proof}

In general, if $M$ is a maximal subgroup of $G$ containing
an element $t\in\mathcal C$, and such that $M\cap \mathcal C$ is an $M$-conjugacy class 
(which occurs in particular in parts (a) and (b) of Lemmas~\ref{c1casessl} 
and~\ref{c1cases}), then the contribution from the conjugacy
class of $M$ to the estimates for the quantities in \eqref{incl:p1}
is given by:
\begin{equation}\label{c1:singleclass}
\frac{|G|}{|M|}\frac{|M\cap \mathcal{C}|^2}{|\mathcal{C}|^2}=
\frac{|M|}{|G|}\frac{|C_G(t)|^2}{|C_M(t)|^2}.   
\end{equation}
We recall the paritiy restrictions on $n$ given by Lemma~\ref{good:cent:order}.

\subsection{Linear groups}\label{c1:linear}

Here $G\cong \SL_{2n}(q)$, $n\geq 2$, $q\geq2$. Our proofs work for all $n,q$ but 
give an upper estimate for $p_1(\cnqsl)$ less than $1$ only for $q\geq3$. 
We believe that there must be substantial overcounting of pairs between cases (a) 
and (c) of Lemma~\ref{c1casessl} but have been unable to improve the estimate
significantly.

\begin{lemma}\label{lem:c1sl}
If $X=\SL_{2n}(q)$ with $n\geq2, q\geq3$,  then
\begin{equation*}
p_1(\cnqsl)< \frac{2}{q-1} - \frac{4}{q^{n+1}}. 
\end{equation*}
\end{lemma}

\begin{proof}

We must consider the maximal reducible subgroups $M$ in Lemma~\ref{c1casessl}.
For each of the parts (a)-(c) there is one
$G$-conjugacy class of subgroups $M$, and since the groups of type $P_1$ and 
$P_{2n-1}$ are interchanged by an outer automorphism of $G$ which leaves classes of good
elements invariant, we have the same contribution to \eqref{incl:p1} from parts (a) and (b). Note
that, for $G\cong \SL_{2n}(q)$, $\mathcal C$ is also a conjugacy class of $\gl_{2n}(q)$, so
\begin{equation}\label{csl}
 |\mathcal C|=\frac{|G|}{|C_G(t)|}=\frac{|\gl_{2n}(q)|}{|\gl_n(q)|(q^n-1)}.
\end{equation}

Suppose first that $M\cong P_1$. The contribution from such subgroups to
$p_1(\cnqsl)$ is given by \eqref{c1:singleclass}.
Since $M \cap\mathcal C$ is a single $M$-conjugacy class, $G$
is transitive on pairs $(Z, t)$ where $t \in\mathcal C$, $Z$ is a 1-space, 
and $t$ leaves $Z$ invariant.
Hence $C_G(t)$ is transitive on the $(q^n - 1)/(q- 1)$ $1$-spaces fixed by $t$, that is,
$|C_G(t)|/|C_M(t)|=(q^n - 1)/(q- 1)$. Thus, since $|G : P_1| = (q^{2n} -1)/(q- 1)$,
the contribution from subgroups $P_1$ is
\begin{align*}
\frac{|G|}{|M|}\frac{|M\cap \mathcal{C}|^2}{|\mathcal{C}|^2}&=
\frac{|M|}{|G|} \frac{|C_G(t)|^2}{|C_M(t)|^2}=
\frac{q-1}{q^{2n}-1}\,\frac{(q^n-1)^2}{(q-1)^2}\\ 
&= \frac{q^n-1}{(q^n+1)(q-1)}
< \frac{1}{q-1} - \frac{2}{q^{n+1}}   - \frac{1}{q^{n+2}}.
\end{align*}
The contribution from the groups $P_{2n-1}$ in part (a) is equal to this, as mentioned
above. Now consider the groups $M = P_n$ stabilising an $n$-space $Z$ as in
Lemma~\ref{c1casessl}~(c)
The pairs $(t, t^g)$ in $M$ we need to count are those where exactly
one of $t$ or $t^g$ fixes $Z$ pointwise. Let $U, W$ be as in Lemma~\ref{c1casessl}.
There are $q^{n^2}$ complementary subspaces $Y$ to $Z$, and for each such $Y$, 
the number of elements $t \in\mathcal C$ such that $(U,W) = (Z, Y )$ is 
$|\gl_n(q)|/(q^n -1)$. Thus the number of $t$
fixing $Z$ pointwise is $q^{n^2}|\gl_n(q)|/(q^n -1)$,
and this is equal to the number of
$t^g$ with $W^g = Z$. Thus the number of pairs $(t, t^g)$ satisfying 
Lemma~\ref{c1casessl}~(c) with
$Z = U = W^g$ is $q^{2n^2}|\gl_n(q)|^2/(q^n -1)^2$, and the 
contribution to \eqref{incl:p1} from these subgroups, using 
\eqref{csl} and then Lemma~\ref{theta:ref}, is
 \begin{align*}
\frac{|G|}{|M|} \frac{2q^{2n^2}|\gl_n(q)|^2}{|\mathcal C|^2(q^n -1)^2}&=
 \frac{|\gl_{2n}(q)|}{|\gl_n(q)|^2q^{n^2}} \frac{2q^{2n^2}|\gl_n(q)|^4}{|\gl_{2n}(q)|^2}\\
&= \frac{2q^{n^2}|\gl_n(q)|^2}{|\gl_{2n}(q)|} = \frac{2\Theta(1,n;q)^2}{q^{n^2}\Theta(1,2n;q)}< \frac{2}{q^{n^2}}.
\end{align*}

Summing the contributions to \eqref{incl:p1}
from all the groups in Lemma~\ref{c1casessl}
we find that the contribution $p_1(\cnqsl)$ satisfies
\begin{equation*}
p_1(\cnqsl)< \frac{2}{q-1} - \frac{4}{q^{n+1}} - \frac{2}{q^{n+2}} +  \frac{2}{q^{n^2}} \leq \frac{2}{q-1} -\frac{4}{q^{n+1}}. 
\end{equation*}
\end{proof}

\subsection{Reducible subgroups with $X\ne\SL$}\label{c1:symplectic}

Throughout this subsection, $t$ is a good element in a $G$-conjugacy class $\mathcal C$, 
and $t$ lies in one of the subgroups $M$ in Lemma~\ref{c1cases}.
We deal uniformly, for all types $X\ne\SL$, with the groups in parts (a), (b) and (c) of Lemma~\ref{c1cases} in three separate lemmas.

\begin{lemma}\label{lem:c1parta}
For $X=\Su, \Sp, \So^\eps$, the contribution to $p_1(\cnqx)$ from the subgroups in Lemma~$\ref{c1cases}~(a)$
satisfies
\begin{align*}
\text{`contribution'}&<\begin{cases}
  \frac{1}{q(q^2-1)} - \frac{3}{4q^{n+2}} &\mbox{if $X=\Su$}\\
 \frac{1}{q-1} - \frac{2}{q^{n+1}} &\mbox{if $X=\Sp$}\\
 \frac{1}{q(q-1)} -\frac{2}{q^{n/2+1}} + \frac{8}{3q^{n}} &\mbox{if $X=\So^+$}\\
 \frac{1}{q(q-1)} +\frac{4}{q^{n/2+1}}                    &\mbox{if $X=\So^-$.}\\
\end{cases}
\end{align*}
\end{lemma}

\begin{proof}
Let $G=X_{2n}(q)$, and consider a subgroup $M = P_1$ stabilising a totally isotropic $1$-space $Z$ as in
Lemma~\ref{c1casessl}~(a) and these subgroups form a single $G$-conjugacy class. 
By that lemma, $M\cap\mathcal C$ is a single $M$-conjugacy class, and its 
elements fix $Z$ pointwise. Thus the contribution to $p_1(\cnqx)$ is given by \eqref{c1:singleclass}.

Suppose first that $X=\Su$, so $n$ is odd. Then by \cite[Props 4.1.4, 4.1.18]{kleidman1990}
$M\cong q^{4n-3}.(\SL_1(q^2) \times \Su_{2n-2}(q)).(q^2-1)$. Also 
$|C_M(t)|= (q^n+1)(q^2-1)q^{2n-3}|\SL_{1}(q^2)||\Su_{n-2}(q)|$.
Thus using \eqref{c1:singleclass} and Tables \ref{order:classical} and \ref{order:cent:good}, 
 the contribution to $p_1(\cnqx)$ is 
\begin{eqnarray*}
\frac{|M|}{|G|}\frac{|C_G(t)|^2}{|C_M(t)|^2}
&=&\frac{(1+q^{-n})^2(1-q^{-n+1})^2}{(1-q^{-2n})(1+q^{-2n+1})(1-q^{-2})} \frac{1}{q^3}\\
&=&\frac{(1+q^{-n})(1-q^{-n+1})^2}{(1-q^{-n})(1+q^{-2n+1})q(q^2-1)}.\\
\end{eqnarray*}
Now $D:=(1-q^{-n})(1+q^{-2n+1})=1-q^{-n}+q^{-2n+1}-q^{-3n+1}<1$, and 
\begin{eqnarray*}
N&:=&  (1+\frac{1}{q^{n}})(1-\frac{1}{q^{n-1}})^2 = 1-\frac{2}{q^{n-1}}+\frac{1}{q^{n}}+ 
\frac{1}{q^{2n-2}} - \frac{2}{q^{2n-1}} + \frac{1}{q^{3n-2}}\\
&=& D - \frac{2}{q^{n-1}}+\frac{2}{q^{n}}+ 
\frac{1}{q^{2n-2}} - \frac{3}{q^{2n-1}} + \frac{1}{q^{3n-2}} +\frac{1}{q^{3n-1}}\\
&<& D-\frac{2}{q^{n-1}} ( 1-\frac{1}{q}-\frac{1}{2q^{n-1}})\leq D-\frac{3}{4q^{n-1}}
\end{eqnarray*}
where the last inequality used $n\geq3, q\geq2$. Thus the contribution is at most
\[
 \frac{D-\frac{3}{4}q^{-n+1}}{Dq(q^2-1)} = \frac{1}{q(q^2-1)} - \frac{3}{4Dq^n(q^2-1)}
<  \frac{1}{q(q^2-1)} - \frac{3}{4q^{n+2}}.
\]

Suppose next that $X=\Sp$, so $n$ is even. Then by \cite[Prop 4.1.19]{kleidman1990}, $M\cong q^{2n-1}.(\gl_1(q) \times \Sp_{2n-2}(q))$. Also 
$|C_M(t)|=q^{n-1}(q^{n/2+1})|\Sp_{n-2}(q)| |\gl_1(q)|$. Thus, using \eqref{c1:singleclass} and Tables \ref{order:classical} and 
\ref{order:cent:good}, the contribution to $p_1(\cnqx)$ is 
\begin{eqnarray*}
\frac{|M|}{|G|}\frac{|C_G(t)|^2}{|C_M(t)|^2}
&=&\frac{(1-q^{-n})^2}{(1-q^{-2n})(1-q^{-1})q}=\frac{(1-q^{-n})}{(1+q^{-n})(q-1)}\\
&=&\frac{1}{q-1}-\frac{2}{(q^n+1)(q-1)}\leq \frac{1}{q-1} - \frac{2}{q^{n+1}}. 
\end{eqnarray*}

Finally suppose that $X=\So^\eps$, so $n$ is even and $n\geq4$. Then  by  \cite[Prop 4.1.20]{kleidman1990}
$M\cong q^{2n-2}(\gl_1(q) \times \So_{2n-2}^\eps(q))$. Also $|C_M(t)|=q^{n-2}(q^{n/2}+1)|\So_{n-2}^{-\eps}(q)||\gl_1(q)|$
Thus, using \eqref{c1:singleclass} and 
Tables \ref{order:classical} and \ref{order:cent:good}
the contribution to $p_1(\cnqx)$ is
\begin{eqnarray*}
\frac{|M|}{|G|}\frac{|C_G(t)|^2}{|C_M(t)|^2}
  &=&\frac{1}{q^{2}}\frac{(1-\eps
  q^{-n+1})(1+\eps q^{-n/2})^2}{(1-\eps q^{-n})(1+\eps
  q^{-n/2+1})^2} \frac{(1-q^{-n+2})^2}{(1-q^{-2n+2})(1-q^{-1})} \\
&=&\frac{1}{q(q-1)} Y(\eps) 
\end{eqnarray*}
where $Y(\eps)=\frac{(1-\eps
  q^{-n+1})(1+\eps q^{-n/2})^2}{(1-\eps q^{-n})(1+\eps
  q^{-n/2+1})^2} \frac{(1-q^{-n+2})^2}{(1-q^{-2n+2})}$.
  
Suppose first that $\eps=+$. Then 
\begin{eqnarray*}
Y(+)  &=&\frac{(1- q^{-n+1})(1+ q^{-n/2})^2}{(1- q^{-n})(1+
  q^{-n/2+1})^2} \frac{(1-q^{-n+2})^2}{(1-q^{-2n+2})} \\
&=& \frac{(1+ q^{-n/2})(1- q^{-n/2+1})^2}{(1- q^{-n/2})(1+
  q^{-n+1})}  \\
&<&\frac{1-2q^{-n/2+1}+q^{-n/2} -2q^{-n+1}+q^{-n+2} +q^{-3n/2-2}}{1-q^{-n/2}}\\
&<& 1 + \frac{-2q^{-n/2}(q-1)+q^{-n+2}}{1-q^{-n/2}}
\end{eqnarray*}
so the contribution is 
\[
 \frac{Y(+)}{q(q-1)}< \frac{1}{q(q-1)} -\frac{2}{q^{n/2+1}} + \frac{8}{3q^{n}}
\]
where for the last term we used the conditions $q\geq2, n\geq4$. Similarly if $\eps=-$, 
\begin{eqnarray*}
Y(-)  &=&\frac{(1+ q^{-n+1})(1- q^{-n/2})^2}{(1+ q^{-n})(1-
  q^{-n/2+1})^2} \frac{(1-q^{-n+2})^2}{(1-q^{-2n+2})} \\
&=& \frac{(1- q^{-n/2})^2(1+ q^{-n/2+1})^2}{(1+ q^{-n})(1-
  q^{-n+1})}  \\
&<&\frac{1+2q^{-n/2+1}-2q^{-n/2} +q^{-n+2}}{1-q^{-n+1}}\\
&=& 1 + \frac{2q^{-n/2}(q-1)+q^{-n+2}+q^{-n+1}}{1-q^{-n+1}}
\end{eqnarray*}
so the contribution is 
\begin{eqnarray*}
 \frac{Y(-)}{q(q-1)}&<& \frac{1}{q(q-1)} +\frac{2}{q^{n/2+1}(1-q^{-n+1})} + \frac{q+1}{q^{n}(q-1)(1-q^{-n+1})}\\
&\leq& \frac{1}{q(q-1)} +\frac{4}{q^{n/2+1}} 
\end{eqnarray*}
again using $q\geq2, n\geq4$.  
\end{proof}

\begin{lemma}\label{lem:c1partb}
For $X=\Su, \So^\eps$, the contribution to $p_1(\cnqx)$ from the subgroups in Lemma~$\ref{c1cases}~(b)$
 is at most 
\[
\frac{1}{q(q+1)} + \frac{2}{q^{n+2}},\quad  \frac{1}{(2,q-1)q} + \frac{8}{3q^{n/2+1}},\quad \frac{1}{(2,q-1)q} 
\]
for $X=\Su, \So^+, \So^-$ respectively.
\end{lemma}

\begin{proof}
Let $G=X_{2n}(q)$, and consider a subgroup $M = N_1$ stabilising a non-degenerate $1$-space $Z$ as in
Lemma~\ref{c1casessl}~(b). By that lemma, $M\cap\mathcal C$ is a single $M$-conjugacy class, and its 
elements fix $Z$ pointwise. Thus the contribution to $p_1(\cnqx)$ is given by \eqref{c1:singleclass}.

Suppose first that $X=\Su$. Then $M\cong (\Su_{1}(q) \times \Su_{2n-1}(q))(q+1)$, and 
$|C_M(t)|=(q^{n}+1)(q+1)|\Su_{n-1}(q)|$. Thus, using \eqref{c1:singleclass} and Tables \ref{order:classical} and 
\ref{order:cent:good}, the contribution (recalling that $n$ is odd and at least 3) is 
\begin{eqnarray*}
\frac{|M|}{|G|}\frac{|C_G(t)|^2}{|C_M(t)|^2}
&=& \frac{(1+q^{-n})^2}{(1-q^{-2n})(1+q^{-1})}\,\frac{1}{q^2}\\
&=& \frac{(1+q^{-n})}{(1-q^{-n})(1+q^{-1})}\,\frac{1}{q^2}=\frac{1}{q(q+1)}(1+\frac{2}{q^n-1})\\
&<& \frac{1}{q(q+1)} + \frac{2}{q^{n+2}}.
\end{eqnarray*}

Now suppose that $X=\So^\eps$.
Then $M\cong (\Og_{1}(q) \times \Og_{2n-1}(q))\cap \So_{2n}^\eps(q)$, and 
$|C_M(t)|=(q^{n/2}+1) |\Og_{n-1}(q)|$.

Thus, using \eqref{c1:singleclass} and Tables \ref{order:classical} and 
\ref{order:cent:good}, the contribution (recalling that $n$ is even and at least 4) is
\begin{eqnarray*}
\frac{|M|}{|G|}\frac{|C_G(t)|^2}{|C_M(t)|^2} = \frac{(1+\eps q^{-n/2})^2}{(2,q-1)(1-\eps q^{-n})} \frac{1}{q}.
\end{eqnarray*}
If $\eps=+$, then this is
\[
\frac{(1+ q^{-n/2})^2}{(2,q-1)q(1- q^{-n})} = \frac{(1+ q^{-n/2})}{(2,q-1)q(1- q^{-n/2})}\leq \frac{1}{(2,q-1)q} + \frac{8}{3q^{n/2+1}}.
\]
If $\eps=-$, then this is
\[
\frac{(1- q^{-n/2})^2}{(2,q-1)q(1+ q^{-n})} < \frac{1}{(2,q-1)q}.
\]
\end{proof}

\begin{lemma}\label{lem:c1partc}
For $X=\Su, \Sp, \So^+$, the contribution to $p_1(\cnqx)$ from the subgroups in Lemma~$\ref{c1cases}~(c)$
 is less than $\frac{c}{q^{\delta n^2}}$, where $c=\frac{3}{2}, 1, \frac{16}{9}$,
respectively. 
\end{lemma}

\begin{proof}
Let $G=X_{2n}(q)$, and consider a subgroup $M = N_n$ stabilising a non-degenerate $n$-space $Z$ as in
Lemma~\ref{c1casessl}~(c). Then $M$ is a direct product of the classical groups induced
on $Z, Z^\perp$, and there is one conjugacy 
class of such subgroups with size $|G|/2|M|$. Also
$M\cap\mathcal C$ is a union of two $M$-conjugacy classes interchanged by $N_G(M)$.
The pairs $(t, t^g)$ in $M$ we need to count are those where exactly
one of $t$ or $t^g$ fixes $Z$ pointwise and the other fixes $Z^\perp$ pointwise. 
The number of pairs $(t,t^g)$ with $t$ fixing $Z$ pointwise is 
$|M:C_M(t)|^2$, and the same number of pairs arises with $t^g$ fixing $Z$ pointwise.
Thus the contribution to $p_1(\cnqx)$ from these subgroups is
\[
\frac{|G|}{2|M|}\frac{2|M|^2}{|C_M(t)|^2}\frac{1}{|\mathcal C|^2}
=\frac{|M|}{|G|}\frac{|C_G(t)|^2}{|C_M(t)|^2}.
\]
For $X=\Su, \Sp$, using Tables \ref{order:classical} and \ref{order:cent:good}, this quantity is equal to 
\[
\frac{1}{q^{2n^2}}\frac{\Theta(1,n;-q)^2}{\Theta(1,2n;-q)},\quad \frac{1}{q^{n^2}}\frac{\Theta(1,n/2;q^2)^2}{\Theta(1,n;q^2)},
\]
respectively, and since
$\frac{\Theta(1,n;-q)^2}{\Theta(1,2n;-q)}<1+\frac{1}{q}\leq \frac{3}{2}$ and
$\frac{\Theta(1,n/2;q^2)^2}{\Theta(1,n;q^2)}=\frac{\Theta(1,n/2;q^2)}{\Theta(n/2+1,n;q^2)}<1$,  by Lemma~\ref{theta:ref}(ii)), 
the result follows for these cases. Now suppose that $X=\So^+$. By Lemma~\ref{c1cases}, $Z$ has minus type, 
so  $|C_M(t)|=|C_G(t)|=(q^{n/2}+1)|\So_n^-(q)|$. By Tables \ref{order:classical} and \ref{order:cent:good} and 
Lemma \ref{theta:ref} (recall $n$ is even here), this contribution is
\begin{eqnarray*}
{}&&\frac{|\Og_n^-(q)|^2}{(2,q-1)|\So_{2n}^+(q)|} 
=\frac{q^{n^2-n}(1+q^{-n/2})^2\Theta(1,n/2-1;q^2)^2}{(2,q-1)q^{2n^2-n}(1-q^{-n})\Theta(1,n-1;q^2)}\\
&\leq&\frac{(q^{n/2}+1)}{q^{n^2}(q^{n/2}-1)} \frac{\Theta(1,n/2-1;q^2)}{\Theta(n/2,n-1;q^2)}< \frac{(q^{n/2}+1)}{q^{n^2}(q^{n/2}-1)(1-q^{-n})}\leq \frac{16}{9q^{-n^2}}\\
\end{eqnarray*}
where we use $n\geq 4$, $q\geq2$ in the last inequality. 
 \end{proof}
 
\subsection{Proof of Theorem~\ref{main}}

The cases where $p_1(\cnqx)=0$ follow from Lemma~\ref{good:cent:order},
and the upper bound for $X=\SL$ is proved in Lemma~\ref{lem:c1sl}. Suppose now that $X\ne\SL$
and $G=X_{2n}(q)$ contains good elements. We add the contributions to $p_1(\cnqx)$ from Lemmas~\ref{lem:c1parta},
\ref{lem:c1partb} and \ref{lem:c1partc}. The results for $X=\Sp$ and $X=\So^\eps$ are straightforward;
for $X=\Su$ we note that $\frac{1}{q(q^2-1)}+\frac{1}{q(q+1)}=\frac{1}{q^2-1}$.

\section{{$\mathbf{C}_2$:} Stabilizers of direct sum decompositions}
\label{c2}

The maximal subgroups $M$ of $G=X_{2n}(q)$ ($q\geq3$) in the Aschbacher class $
\mathbf{C}_2$ are stabilizers of direct decompositions $V=V_1\oplus\dots\oplus V_\ell$ with each $V_i$ of dimension
$d:=2n/\ell$ and $\ell\geq2$. Thus $M\leq \hat M:=\GL_d(q^\delta)\wr S_\ell$. Suppose that $M\cap\mathcal C\ne\emptyset$ for a 
$G$-conjugacy class $\mathcal C$ of good elements and let $t\in M\cap\mathcal C$ and $g\in G$ be such that $\la t,t^g\ra\leq M$ and 
is irreducible on $V$. If $(n,q)\ne (2,2^a-1)$ for any $a$, and $(n,q)\ne(6,2)$, let $r$ be a $\ppd(n;q^\delta)$ prime dividing $o(t)$, and recall that $r\geq n+1$.
\emph{In this section we will assume that $\la t,t^g\ra$ cannot be written over a proper subfield: those subgroups not satisfying 
this condition are contained in $\mathbf{C}_5$ subgroups and will be treated in Section~\ref{c5}.}
First we restrict $d, \ell, r, X$. (Note that the excluded prime powers of the form $q=2^a-1$ are precisely the Mersenne primes.)

\begin{lemma}\label{c2cases}
Suppose that $(n,q)\ne (2,5)$, or $(2,2^a-1)$ for any $a$. Then $d=1$, $\ell=2n\geq8$, $(n,q)\ne(6,2)$, $r=n+1$, and $X\in\{\SL, \So^+\}$.
Moreover, if  $X=\So^+$, then $q$ is an odd prime. 
\end{lemma}
 
\begin{proof}
Suppose first that $d=1$ so $\ell=2n$. Then $q^\delta>2$, since otherwise $M\cong S_{2n}$ which is not irreducible on $V$,
contradicting the fact that  $\la t,t^g\ra$ is irreducible and contained in $M$. In particular $(n,q)\ne(6,2)$ (for if
$n$ is even then $X\ne\Su$ and so $\delta=1$). Similarly if $n=2$ then $\delta=1$ since $\Su_4(q)$ does not contain good elements, so $q>2$; in this case,
 $|M|$ divides $24(q-1)^4$, and $M$ contains an element of
order $o(t)=(q+1)/(2,q-1)$. If $q$ is even this implies that $q=q^\delta=2$, which is a contradiction, so $q$ is odd. 
Now $t$ induces a permutation of $\{V_1, V_2, V_3, V_4\}$ of order $s\leq 4$, and $t^s\in\GL_1(q)^4$ has order
at most $(o(t),q-1)\leq 2$. Thus $o(t)=s$ or $2s$, and $q+1=2s$ or $4s$ respectively, so $q\in\{3,5,7,11\}$.
Our assumptions imply that $q=11$ so $o(t)=2s=6$. However then  the fixed point space of $t^3$ (which is $t$-invariant 
and contains the fixed point space of $t$) has dimension $0, 1$ or $3$, a contradiction.
Thus $n>2$ and $o(t)$ is divisible by a $\ppd(n;q^\delta)$ prime 
$r$, as defined above. This implies that $r\leq \ell$ so $r=n+1>3$, and in particular $n$ is even,
so $X\ne\Su$ and $\delta=1$ (see Table~\ref{order:cent:good}). Also $\ell=2n\geq8$.
By \cite[Table 4.2A]{kleidman1990}, 
$X\ne\Sp$, since in the case $\Sp$, the $V_i$ are non-degenerate. Suppose that $X=\So^\eps$. Then 
$q$ must be prime since otherwise $M$ can be written over a proper subfield 
(see \cite[p.100 Remarks on the conditions]{kleidman1990}). Then, since $q=q^\delta>2$, $q$ is odd. 
Then \cite[Prop. 4.2.15(O)]{kleidman1990} implies that $\eps=+$.

Now we assume that $d\geq2$.  
If $(n,q)=(6,2)$ then $X\ne\Su$ so $q=q^\delta=2$ (since $n$ is even)
and $o(t)=9$. It follows that $S_\ell$ contains an element of order $9$, 
which is impossible since $\ell=2n/d\leq n=6$.  
By our assumptions on $(n,q)$ it now follows that
$o(t)$ is divisible by a $\ppd(n;q^\delta)$ prime $r$ as defined above.
Now $\ell\leq n < r$. Let $t_0:=t^{o(t)/r}$. Then $t_0\in\GL_d(q^\delta)^\ell$ 
fixes each $V_i$ setwise.   Since $r$ is a  $\ppd(n;q^\delta)$ prime, it follows that $d=n$ and $\ell=2$.
Moreover, we may assume that $t_0$ acts irreducibly on $V_1$ and fixes $V_2$ pointwise. Then also $t$ must 
fix $V_1$ pointwise and hence leave $V_1$ and $V_2$ invariant. Similarly $t^g$ leaves each
of $V_1$ and $V_2$ invariant, which implies that $\la t,t^g\ra$ is reducible, a contradiction. 
\end{proof}

We prove the following estimates.

\begin{lemma}\label{lem:c2probs}
Suppose that  $(n,q)\ne (2,5)$, or $(2,2^a-1)$ for any $a$. 
\begin{enumerate}
 \item[(i)]  Then $p_2(\cnqx)=0$ if $X=\Su, \Sp$ or $\So^-$, and  
\[
 p_2(\cnqx) < \begin{cases}
               \frac{1}{q^{n^2}}  &\mbox{if $X=\SL$} \\
	       \frac{1}{2q^{n^2-3n}} 	&\mbox{if $X=\So^+$.}\\
\end{cases}
\]
\item[(ii)] $\tilde{p}_2(\cnqsp)=0$.
\end{enumerate}

\end{lemma}

\begin{proof}
Suppose that $G\cap\mathcal{C}\ne\emptyset$. Then,
by Lemma~\ref{c2cases}, in all cases $o(t)$ is divisible by a $\ppd(n;q)$ prime $r=n+1$, and 
the subgroups $\la t,t^g\ra$ to be considered
lie in a conjugate of the subgroup $N=\GL_1(q)\wr S_{2n}$ of $\GL(V)$,
permuting the $2n$ coordinate vectors $e_1,\ldots, e_{2n}$. 
The Sylow $r$-subgroups  of $N$ are cyclic, so $t^{o(t)/r}$ can be conjugated by some $x \in G$ to a permutation matrix 
$s$ that permutes the first $n+1$ coordinate vectors $e_1,\ldots, e_{n+1}$ and fixes the last $n-1$ coordinate vectors 
$e_{n+2},\ldots, e_{2n}$. 
Thus the fixed point space of $s$ has basis $e_1+\cdots +e_{n+1}$, $e_{n+2},\ldots, e_{2n}$, and 
this must also be a basis of the fixed point space of $t^x$. Now $t^x$ must induce an $r$-cycle in $S_{2n}$
(since $t^x$ centralises the $r$-cycle $(t^{o(t)/r})^x$), and so $t^x=\mathbf{y}\sigma$ where
$\mathbf{y}=(y_i)\in(\mathbb{F}_q^*)^{2n}$ and, without loss of 
generality, $\sigma=(1,2,\dots,r)\in S_{2n}$.
Since $t^x$ fixes $e_i$, for $i>r$, we have $y_i=1$, for $i>r$. Also, since $t^x$ fixes $\sum_{i=1}^re_i$, 
it follows that $y_i=1$ for $i\leq r$ also. Thus $t^x=s\in S_{2n}$.
It follows that $|C_N(t)|=(q-1)^{n}(n-1)!(n+1)$. 
Also, the cyclic torus $T^x$ containing $t^x$, given by Lemma~\ref{good:cent:order}, satisfies $T^x\cap N\leq 
Z_{q-1}\times \la t^x\ra$. and since $r=n+1$ is prime, $N_N(T^x\cap N)/C_N(T^x\cap N)\cong Z_n$. It follows that
$N\cap\mathcal C$ is a single $N$-conjugacy class. (Indeed this argument shows that in all the cases below
$M\cap\mathcal C$ is a single $M$-conjugacy class, so that $p_2(\cnqx)$ is given by \eqref{prob:cc}.)

Suppose first that $X=\SL$. Then the maximal $\mathbf{C}_2$ subgroups $M$ containing the groups $\la t,t^g\ra$
form a single $G$-conjugacy class, have index $q-1$
in some conjugate of $N$, and $|C_M(t)|=|C_N(t)|/(q-1)$.  Thus $p_2(\cnqsl)$ is equal to
\begin{eqnarray*}
\frac{|M|}{|G|}\frac{|C_G(t)|^2}{|C_M(t)|^2}&=&\frac{(q-1)^{2n-1}(2n)!}{q^{n(2n-1)} \prod_{i=2}^{2n}(q^i-1)} \frac{(q^n-1)^2q^{n(n-1)} \prod_{i=2}^n (q^i-1)^2}{(n-1)!^2 (q-1)^{2n-2} (n+1)^2}\\
&<&\frac{(2n)!}{(n-1)!^2(n+1)^2} \frac{\prod_{i=1}^n(q^i-1)}{q^{n^2-2n}\prod_{i=n+1}^{2n}(q^i-1)}.
\end{eqnarray*}
Since $\frac{q^i-1}{q^{n+i}-1}<\frac{1}{q^n}$ for each $i\geq 1$, and $\frac{(2n)!}{(n-1)!^2(n+1)^2}=
\frac{(2n)!}{(n)!^2}\frac{n^2}{(n+1)^2}<{2n \choose n}<4^n$, we have 
$p_2(\cnqsl) < 4^n/q^{2n^2-2n}$.
Since $n \geq 4$, by Lemma~\ref{c2cases}, we have $q^{n^2}\geq q^{4n}\geq 4^nq^{2n}$, so
$p_2(\cnqsl) <1/q^{n^2}$.

Now suppose that $X=\So^+$, with $q$ an odd prime and $n\geq4$ even, as in Lemma~\ref{c2cases}. Then $M\cong (O_1(q)\wr S_{2n})\cap \So_{2n}^+(q)$, 
so $|M|=2^{2n-1}(2n)!$ and 
$|C_M(t)|=2^{n-1}(n-1)!(n+1)$. By \cite[Prop. 4.2.15]{kleidman1990}, the number of $G$-conjugacy classes of such subgroups
$M$ is $c\leq 4$.
Thus $p_2(\cnqso)$ is equal to
\begin{eqnarray*}
c\,\frac{|M|}{|G|}\frac{|C_G(t)|^2}{|C_M(t)|^2}&=&\frac{c2^{2n-1}(2n)!}{|\So_{2n}^{+}(q)|} \frac{(q^{n/2}+1)^2|\So_n^{-}(q)|^2}{(n-1)!^2 2^{2n-2} (n+1)^2}\\
&\leq&\frac{8(2n)!}{(n-1)!^2(n+1)^2} \frac{(q^{n/2}+1)^2|\So_n^{-}(q)|^2}{|\So_{2n}^+(q)|}.
\end{eqnarray*}
As above, $\frac{(2n)!}{(n-1)!^2(n+1)^2}= \frac{(2n-2)!}{(n-1)!^2}\frac{2n(2n-1)}{(n+1)^2}<4^n$. 
Moreover, by Table \ref{order:classical} and Lemma \ref{theta:ref}(iii),
\begin{eqnarray*}
\frac{(q^{n/2}+1)^2|\So_n^{-}(q)|^2}{|\So_{2n}^+(q)|}&=&\frac{(q^{n/2}+1)^2 q^{n^2-n}
(1- q^{-n/2})^2 \Theta(1,n/2-1;q^2)^2}{q^{2n^2-n}(1- q^{-n})\Theta(1,n-1;q^2)}\\
&<&\frac{(q^{n/2}+1)^2(1- q^{-n/2})^2}{q^{n^2}(1- q^{-n})^2}
=\frac{1}{q^{n^2-n}}.
\end{eqnarray*}
Thus $p_2(\cnqso)<\frac{8.4^n}{q^{n^2-n}}$ which is less than
$\frac{1}{2q^{n^2-3n}}$ since $n\geq4$ and $q\geq3$. This proves part (i).

Now we consider $X=\Sp$ with $q$ even, and the proportion $\tilde{p}_2(\cnqsp)$ given by \eqref{main3:prob}.
By Lemma~\ref{c2cases} there are no subgroups to consider since $q$ is even. Hence   
$\tilde{p}_2(\cnqsp)=0$.
\end{proof}

\section{{$\mathbf{C}_3$:} Stabilizers of extension fields} 
\label{c3}

Let $M$ be a maximal subgroup of $G=X_{2n}(q)$ belonging to Aschbacher class $\mathbf{C}_3$, where $n\geq3$. 
Then $M$ preserves on $V$ the structure 
of an $\mathbb{F}_{q^{\delta s}}$-space of dimension $2n/s$, for some prime $s$, so
 $M\leq\GL_{2n/s}(q^{\delta s}).s < \GL(V)$.  Let $\mathcal C$ be a class of good elements,
and suppose that $M\cap\cal C\ne\emptyset$. Let $t\in M\cap \cal C$ and, if $n\ne 2$ and $(n,q)\ne (6,2)$, let $r$ be a 
$\ppd(n,q^\delta)$-prime dividing $o(t)$. First we derive some basic facts about $M\cap\mathcal C$ and
identify the possibilities for $M$ using 
\cite[Table 4.3.A]{kleidman1990}.

\begin{lemma}\label{c3cases} Suppose that $(n,q)\ne (2,3), (2,7)$. 
The dimension $n=ds$, for some $d$, $M\cap\mathcal{C} \subseteq G\cap\GL_{2d}(q^{\delta s})$,
and the the number $c$ of $G$-conjugacy classes for 
$M$ and the possible structure of $M$ are given in Table~\ref{tab:c3}. Also 
$M\cap\mathcal{C}$ is a single $M$-conjugacy class, or possibly splits into two equal sized 
classes in Line $6$ of Table~$\ref{tab:c3}$ when $q$ is odd.
\end{lemma}

\begin{table}[h]
 \begin{center}
\begin{tabular}{l|l|l|l}\hline
$X$	& $M$ &$c$&Conditions\\ \hline
$\SL$	& $(\gl_{2d}(q^s).s)\cap \SL_{2n}(q)$ &1&\\
$\Su$	& $(\gu_{2d}(q^s).s)\cap \Su_{2n}(q)$ &1& $d, s$ odd\\
$\Sp$	& $\Sp_{2d}(q^s).s$                   &1& $d$ even\\
$\Sp$	& $\gu_{n}(q).2$            &1&  $s=2$ and $d$, $q$ odd\\
$\So^\eps$   & $(\Og_{2d}^\eps(q^s).s)\cap \So_{2n}^\eps(q)$ &$1$ or $(s,2)$& $d$ even, \\ 
$\So^+$   & $\gu_{n}(q)\cap \So_{2n}^+(q)$ &$2$&  $s=2$ and $d$ odd \\ \hline
\end{tabular}   
 \end{center}
\caption{$\mathbf{C}_3$ subgroups containing  good elements.}\label{tab:c3}
\end{table}

\begin{proof}
If $o(t)$ is divisible by a $\ppd(n,q^\delta)$-prime $r$, then, by 
Lemma~\ref{good:cent:order}, the element $t_0:=t^{o(t)/r}$ of order $r$ 
lies in a cyclic torus $T$ of $G$ of order dividing $q^{\delta n}-1$ 
such that $T\subset M\cap\GL_{2n/s}(q^{\delta s})$ and $T$ has an  
$n$-dimensional $\mathbb{F}_{q}$-fixed point subspace in $V$. 
In particular the group $\GL_{2n/s}(q^{\delta s})$ has non-cyclic Sylow $r$-subgroups, 
and this is true if and only if $s$ divides $n$. Let $d:=n/s$. 
If the good element $t$ does not lie in $M_0:= G\cap\GL_{2d}(q^{\delta s})$, 
then $t$ induces a nontrivial field automorphism of $M_0$, and in particular $t$ does not 
centralise $\la t\ra\cap M_0$, which is a contradiction. Thus $t\in M_0$ and hence 
$M\cap{\cal C}\subset M_0$. 
A similar argument shows that $t\in M_0$ in the cases $n=2$ and $(n,q)=(6,2)$.
The possibilities for $M$ are given in \cite[Table 4.3.A]{kleidman1990}
(note that $n$ is even if $X=\So^\eps$, so Line 6 of \cite[Table 4.3.A]{kleidman1990}
is not possible, and in Line 6 of Table~\ref{tab:c3} we must have $\eps=+$). Also 
the values for $c$ come from the various results in \cite[Section 4.3]{kleidman1990}
(the value for $c$ in line 5 depends on $\eps$). 

In all cases the cyclic torus $T$ containing $t$, defined in Lemma~\ref{good:cent:order},
is contained in $M_0$, and all cyclic tori of $M$ of order 
$|T|$ are conjugate in $M_0$. In all cases, except possibly Line 6,  
 we have  $N_M(T)=N_{M_0}(T).s=C_{M_0}(T).n$
and it follows that the $n$ conjugates of $t$ in $T$ (see Lemma~\ref{good:cent:order}) 
are conjugate in $N_M(T)$. Thus $M\cap\cal C$ is a single $M$-conjugacy class.
In Line 6, possibly $N_M(T)=N_{M_0}(T).s=C_{M_0}(T).[n/2]$ and 
$M\cap\cal C$ splits into two equal sized $M$-conjugacy classes.

Finally we verify the conditions column of Table~\ref{tab:c3}. 
If $X=\Su$ then $n$ is odd so both $d, s$ are odd also.  
In Line 3, the fixed point subspace of $t$ is a nondegenerate
$d$-dimensional $\mathbb{F}_{q^{ s}}$-space, and hence $d$ is even. 
In Lines 4 and 6, $s=2$ and a cyclic torus  of $M_0$ with  
$d$-dimensional $\mathbb{F}_{q^2}$-fixed point subspace has order 
$q^{d}- (-1)^{d}$, and since $t$ lies in such a torus it follows that
$d$ is odd. (In Line 4, the condition $q$ odd follows from 
\cite[Table 4.3.A]{kleidman1990}.) 
In Line 5, since $M$ contains $T$ with 
$r\mid |T|$ acting irreducibly on a $d/s$-dimensional 
$\mathbb{F}_{q^s}$-subspace, $d$ must be even.
\end{proof}

Now we consider all these cases to estimate $p_3(\cnqx)$ and $\tilde{p}_3(\cnqsp)$.

\begin{lemma}\label{lem:c3probs}
 Suppose that $(n,q)\ne (2,3), (2,7)$.  
\begin{enumerate}
 \item[(i)] Then
\[
p_3(\cnqx)<\begin{cases}
            \frac{3.6}{q^{n^2}}  &\mbox{if $X=\SL$}\\
           \frac{5}{q^{4n^2/3}}  &\mbox{if $X=\Su$}\\
            \frac{3.7}{q^{n^2/2}}  &\mbox{if $X=\Sp$}\\
            \frac{10.6}{q^{n^2/2}}  &\mbox{if $X=\So^\eps$}.\\
           \end{cases} 
\]
\item[(ii)]  $\tilde{p}_3(\cnqsp)\leq \frac{3.7}{q^{n^2/2}}$.
\end{enumerate}

\end{lemma}

\begin{proof}
\emph{Cases $X=\SL$ and $\Su$.}\quad  
Here there is a unique class of subgroups $M$ for each prime divisor $s$ of $n$,
as in Line 1 or 2, respectively, of Table~\ref{tab:c3}, and 
\[
|C_M(t)|=\begin{cases}
((q^s)^{n/s}-1)|\SL_{n/s}(q^s)|\frac{(q^s-1)}{(q-1)} &\mbox{if $X=\SL$}\\
 (q^n+1)|\gu_{n/s}(q^s)|(\frac{1}{q+1}) &\mbox{if $X=\Su$.}         
         \end{cases}
\]
First let $X=\SL$. By Tables \ref{order:classical} and \ref{order:cent:good}, we have

\begin{eqnarray*}
\frac{|M|}{|G|}\frac{|C_G(t)|^2}{|C_M(t)|^2}&=&
\frac{\Theta(1,\frac{2n}{s};q^s)\Theta(1,n;q)^2}{\Theta(1,2n;q)\Theta(1,\frac{n}{s};q^s)^2} 
\frac{s}{q^{2n^2(1-1/s)}}.
\end{eqnarray*}
By Lemma~\ref{theta:ref}(ii), $\frac{\Theta(1,n;q)^2}{\Theta(1,2n;q)}\le1$, so

\begin{eqnarray*}
\frac{\Theta(1,\frac{2n}{s};q^s)\Theta(1,n;q)^2}{\Theta(1,2n;q)\Theta(1,\frac{n}{s};q^s)^2}
\le \frac{\Theta(\frac{n}{s}+1,\frac{2n}{s};q^s)}{\Theta(1,\frac{n}{s};q^s)}\le \frac{16}{11}
\end{eqnarray*}
where the second inequality follows from Lemma \ref{theta:ref}(ii) since $q^s \ge 4$, so

\begin{eqnarray*}
p_3(\cnqsl) &\leq& \frac{16}{11}\sum_{s\mid n, \ s \ prime} \frac{s}{q^{2n^2(1-1/s)}}
= \frac{16}{11q^{2n^2}}\sum_{s \mid n, \ s \ prime}sq^{2n^2/s}.
\end{eqnarray*}
Since the number of different prime divisors of $n$ is at most $\log n$, by Lemma \ref{calc:c3} we have
$p_3(\cnqsl) \leq \frac{32}{11q^{n^2}}+ \frac{48 \log n}{11q^{4n^2/3}}$.
If the second term arises then $n\geq 3$ and since 
$\frac{\log n}{q^{n^2/3}}\leq \frac{\log n}{2^{n^2/3}}\leq \frac{\log 3}{2^{3}}< 0.138$,
it follows that $p_3(\cnqsl) \leq \frac{3.6}{q^{n^2}}$.

Now let $X=\Su$, so $n, s$ are odd. By Tables \ref{order:classical} and \ref{order:cent:good}, we have
\begin{eqnarray*}
\frac{|M|}{|G|}\frac{|C_G(t)|^2}{|C_M(t)|^2}&=&
\frac{\Theta(1,\frac{2n}{s};-q^s)\Theta(1,n;-q)^2}
{\Theta(1,2n;-q)\Theta(1,\frac{n}{s};-q^s)^2} \frac{s}{q^{2n^2-2n^2/s}}.
\end{eqnarray*}
Since $n/s$ is odd, by Lemma \ref{theta:2} (i), $\frac{\Theta(1,\frac{2n}{s};-q^s)}{\Theta(1,\frac{n}{s};-q^s)}\leq1$, 
and by Lemma \ref{theta:ref}, $\Theta(1,\frac{n}{s};-q^s)>1$. Also, by Lemma \ref{theta:ref}, 
\[
\frac{\Theta(1,n;-q)^2}{\Theta(1,2n;-q)}=\frac{\Theta(1,n;-q)}{\Theta(n+1,2n;-q)} <\frac{(1+q^{-1})}{(1-q^{-n-1})} \leq \frac{8}{5}.
\]
Thus $\frac{|M|}{|G|}\frac{|C_G(t)|^2}{|C_M(t)|^2} < \frac{8}{5}\frac{sq^{2n^2/s}}{q^{2n^2}}.$
Arguing as in the $\SL$-case, and noting that here $s\geq3$, we obtain
\begin{eqnarray*}
p_3(\cnqsu)\leq \frac{24}{5q^{4n^2/3}} + \frac{8\log n}{q^{8n^2/5}} =
\frac{1}{q^{4n^2/3}}(\frac{24}{5} + \frac{8\log n}{q^{4n^2/15}})< 
 \frac{5}{q^{4n^2/3}}
\end{eqnarray*}
noting that, if the second term is present then $n\geq5$.

\medskip\noindent
\textit{Case of Lines $4$ and $6$ of Table~$\ref{tab:c3}$.}\quad We note that $n\equiv 2\pmod{4}$
(since $d$ is odd). We do not restrict $q$ to be odd in Line 4
as we need the estimates for $q$ even when dealing with $\tilde{p}_3(\cnqsp)$.  There are one or two
$G$-conjugacy classes of these subgroups $M$ for $X=\Sp, \So^+$ respectively; also 
$M\cap\cal C$ is either a single $M$-class, or possibly a union of two equal sized $M$-classes when $X=\So^+$. 
Thus the contribution to $p_3(\cnqx)$ is given by \eqref{c1:singleclass} when $X=\Sp$, and 
this must be modifed a bit for $X=\So^+$.  
In both cases, by Table~\ref{order:cent:good}, $|C_M(t)|=(q^{n/2}+1)|\gu_{n/2}(q)|$.

First let $X=\Sp$. Then by Tables~\ref{order:classical} and 
\ref{order:cent:good} and by Lemma~\ref{theta:ref}(ii) (since $n/2$ is odd),  the contribution is
\begin{eqnarray}\label{c32ii}
\frac{|M|}{|G|}\frac{|C_G(t)|^2}{|C_M(t)^2|}&=&\frac{2\Theta(1,n;-q)\Theta(1,\frac{n}{2};q^2)^2}{\Theta(1,n;q^2)
\Theta(1,\frac{n}{2};-q)^2} \frac{1}{q^{n^2/2}}< \frac{2}{q^{n^2/2}}.
\end{eqnarray}
Now let $X=\So^+$. We require one factor of $2$ since there are two $G$-conjugacy classes, and an additional factor of 
$4$ because there may be two $M$-conjugacy classes in $M\cap\mathcal C$. Thus the contribution 
in this case (noting that $n\equiv 2\pmod{4}$  and applying Lemma~\ref{theta:ref}(ii) and (iii))  is at most 
\begin{eqnarray}
\label{c34ii}
8\frac{|M|}{|G|}\frac{|C_G(t)|^2}{|C_M(t)^2|}
&=&\frac{\Theta(1,n;-q)\Theta(1,\frac{n}{2}-1;q^2)^2}{\Theta(1,n-1;q^2)
\Theta(1,\frac{n}{2};-q)^2} \frac{(1+ q^{-n/2})^2}{(1- q^{-n})} \frac{8}{q^{n^2/2}}\nonumber \\
&\leq& \frac{8}{(1-q^{-n/2})^2q^{n^2/2}}<\frac{10.5}{q^{n^2/2}}
\end{eqnarray}
using $q^n\geq 2^6$.

\medskip\noindent
\textit{Case $X=\Sp$.}\quad First we deal with Line 3 of Table~\ref{tab:c3}. 
There is a unique class of subgroups $M$ for each prime divisor $s$ of $n$,
such that $n/s$ is even (so $s=2$ arises only if $n\equiv 0\pmod{4}$),
and $M\cap\mathcal C$ is a single $M$-conjugacy class. Here
$|C_M(t)|=((q^{s})^{\frac{n}{2s}}+1)|\Sp_{n/s}(q^s)|$, and 
by Tables \ref{order:classical} and \ref{order:cent:good}, the contribution from these groups is
\begin{eqnarray*}
\frac{|M|}{|G|}\frac{|C_G(t)|^2}{|C_M(t)|^2}&=&
\frac{\Theta(1,\frac{n}{s};q^{2s})\Theta(1,\frac{n}{2};q^2)^2}
{\Theta(1,n;q^2)\Theta(1,\frac{n}{2s};q^{2s})^2} \frac{s}{q^{n^2-n^2/s}}.
\end{eqnarray*}
By Lemma~\ref{theta:ref}(ii), 
\[
\frac{\Theta(1,\frac{n}{s};q^{2s})\Theta(1,\frac{n}{2};q^2)^2}
{\Theta(1,n;q^2)\Theta(1,\frac{n}{2s};q^{2s})^2} < 
\frac{\Theta(\frac{n}{2s}+1,\frac{n}{s};q^{2s})}
{\Theta(1,\frac{n}{2s};q^{2s})}<\frac{16}{11}
\]
so that the contribution from all of these subgroups is at most 
$ \frac{16}{11}\sum_{s}\frac{s}{q^{n^2(1-1/s)}}$. If  $n\equiv 0\pmod{4}$
so that $s=2$ is allowed, this is less than 
$\frac{32}{11q^{n^2/2}} + \frac{48\log n}{11q^{2n^2/3}}= \frac{1}{q^{n^2/2}}
(\frac{32}{11}+ \frac{48\log n}{11q^{n^2/6}})< \frac{3}{q^{n^2/2}} $ noting that $n\geq 12$
if the second term is present. If $n\equiv 2\pmod{4}$, the contribution is at most
$\frac{48\log n}{11q^{2n^2/3}}
\leq  \frac{1}{q^{n^2/2}} \frac{48\log 3}{11\, 2^{3^2/6}}< \frac{1.7}{q^{n^2/2}}$.
Combining with \eqref{c32ii} in the case $n\equiv 2\pmod{4}$, we see that, for all $n$,
we have $p_3(\cnqsp)< \frac{3.7}{q^{n^2/2}}$.

\medskip\noindent
\textit{Case $X=\So^\eps$.}\quad First we deal with Line 5 of Table~\ref{tab:c3}. 
There are at most $(s,2)$\ $G$-conjugacy classes of subgroups $M$ for each prime divisor $s$ of $n$,
such that $n/s$ is even (so $s=2$ arises only if 
$n\equiv 0\pmod{4}$), and $M\cap\mathcal C$ is a single $M$-conjugacy class. Here
$|M|=s | \So_{2n/s}^\eps(q^s)| \mbox{ and }  |C_M(t)|=((q^s)^{\frac{n}{2s}}+1)|\So^{-\varepsilon}_{n/s}(q^s)|$, so
by Tables \ref{order:classical} and \ref{order:cent:good}, this contribution is at most
\begin{eqnarray*}
(2,s)\frac{|M|}{|G|}\frac{|C_G(t)|^2}{|C_M(t)|^2}
&=&\frac{\Theta(1,\frac{n}{s}-1;q^{2s})\Theta(1,\frac{n}{2}-1;q^{2})^2}
{\Theta(1,n-1;q^{2})\Theta(1,\frac{n}{2s}-1;q^{2s})^2} \frac{(2,s)s}{q^{n^2(1-1/s)}}.
\end{eqnarray*}
Now
\[
 \frac{\Theta(1,\frac{n}{s}-1;q^{2s})\Theta(1,\frac{n}{2}-1;q^{2})^2}
{\Theta(1,n-1;q^{2})\Theta(1,\frac{n}{2s}-1;q^{2s})^2} =
 \frac{\Theta(\frac{n}{2s}+1,\frac{n}{s}-1;q^{2s})\Theta(1,\frac{n}{2}-1;q^{2})}
{\Theta(1,\frac{n}{2s}-1;q^{2s})\Theta(\frac{n}{2}+1,n-1;q^{2})} 
\]
and, by Lemma~\ref{theta:ref}(ii) and (iii), we have 
$\frac{\Theta(1,\frac{n}{2}-1;q^{2})}
{\Theta(\frac{n}{2}+1,n-1;q^{2})} <(1-q^{-n})^{-1}$, and  
\[
(2,s)\frac{|M|}{|G|}\frac{|C_G(t)|^2}{|C_M(t)|^2} < \frac{1-q^{-ns}}{(1-q^{-2s}-q^{-4s})(1-q^{-n})}
\frac{(2,s)s}{q^{n^2(1-1/s)}}=\frac{\nu s}{q^{n^2(1-1/s)}}.
\]
If $s=2$ then $q^s\geq 4$ and $n\geq4$ (since $n/s$ is even), so $\nu\leq 2.143$ if $s=2$;  and 
if $s\geq3$ then $n\geq 2s\geq6$ and $\nu< 1.09$.
Thus the contribution from all of these subgroups is at most $ \frac{4.286}{q^{n^2/2}}+\sum_{s>2}
\frac{(1.09)s}{q^{n^2(1-1/s)}}$, where the first term is present only if  $n\equiv 0\pmod{4}$,
and the second term sums to at most $\frac{3.3\log n}{q^{2n^2/3}}$ (arguing as above). We note that 
the latter function is $\frac{3.3}{q^{n^2/2}}\frac{\log n}{q^{n^2/6}}$, and $\frac{\log n}{q^{n^2/6}}
\leq \frac{\log n}{2^{n^2/6}}\leq \frac{\log n_0}{2^{n_0^2/6}}$ if $n\geq n_0$.
If $n=4, 8$ then the only term is for $s=2$ so the contribution is less than $\frac{4.3}{q^{n^2/2}}$.
If $n\equiv 0\pmod{4}$ and $n\geq12$, then the second term is at most $\frac{3.3}{q^{n^2/2}}\frac{\log 12}{2^{24}}
< \frac{0.001}{q^{n^2/2}}$, so the contribution again is at most $\frac{4.3}{q^{n^2/2}}$.
In $n\equiv 2\pmod{4}$ and $n\geq 6$, then the contribution is the second term and is at most 
$\frac{3.3}{q^{n^2/2}}\frac{\log 6}{2^{6}}
< \frac{0.1}{q^{n^2/2}}$.
Combining with \eqref{c34ii} in the case where $\eps=+$ and $n\equiv 2\pmod{4}$, we see that, for all $n$,
$p_3(\cnqsog)< \frac{10.6}{q^{n^2/2}}$.

\medskip\noindent
\textit{Case $\tilde{p}_3(\cnqsp)$.}\quad Here $X=\Sp$ and $q$ is even. 
By Lemma~\ref{c3cases}, the subgroups to be considered are those in Line 4 of Table~\ref{tab:c3}
(which cover those in Line 5 for $\eps=\pm$), and in the case $n\equiv 2\pmod{4}$
also those in Line 6 (which do not lift to maximal
subgroups in $\mathbf{C}_3$ for $G$). The contribution $\tilde{p}_3(\cnqsp)$,
as given by \eqref{main3:prob}, is therefore equal to the estimate given in the case $\Sp$ above. 
Thus $\tilde{p}_3(\cnqsp)< \frac{3.7}{q^{n^2/2}}$.
\end{proof}

\section{{$\mathbf{C}_4$:} Stabilizers of tensor products}
\label{c4}

We deal with maximal subgroups in $\mathcal{C}_4$ for $n\geq3$.

\begin{lemma}\label{lem:c4probs} Suppose that $n\geq3$. Then, for
 all $X$, $p_4(\cnqx)=\tilde{p}_4(\cnqsp)=0$.
\end{lemma}

\begin{proof}
By \cite[Table 4.4.A]{kleidman1990}, for all $X$, if $M$ is a maximal $\mathbf{C}_4$ subgroup in $G=X_{2n}(q)$,
then $M$ is the stabilizer  in $G$ of a tensor decomposition $V=U\otimes W$ with $\dim(U)\geq \dim(W)\geq2$,
and $M\leq \GL(U)\otimes\GL(W)$. 
Thus, since $n\geq3$, $M$ contains elements of order belonging to $\Phi^X(n,q)$ if and only if 
$\dim(U)=n, \dim(W)=2$. Let $t$ be an element of such a subgroup
with $o(t)\in \Phi^X(n,q)$, say $t=t_U\otimes t_W$, with $t_U\in\GL(U), t_W\in\GL(W)$.
Then $t^{q^2-1}=t'\otimes 1$ with order $o(t')\in \Phi^X(n,q)$. Let $\{w_1,w_2\}$ be a basis for $W$,
and for $j=1,2$, let $V_j:= \{ u\otimes w_j \mid u\in U\}$. Then each $V_j$ is $(t'\otimes 1)$-invariant 
and $t'\otimes 1$ acts irreducibly on it. Moreover, $V=V_1 \oplus V_2$, and hence $t'\otimes1$ does not 
fix an $n$-dimensional subspace of $V$ pointwise, and so is not a good element. It follows that 
$M$ does not contain good elements.
\end{proof}

\section{{$\mathbf{C}_5$:} Stabilizers of subfields} 
\label{c5}

Let $M$ be a maximal subgroup of $G=X_{2n}(q)$ belonging to Aschbacher class $\mathbf{C}_5$. 
By \cite[Table 4.5.A]{kleidman1990}, $M$ is  a cyclic extension of $X'_{2n}(q_0)$, for some type $X'$, 
and $M$ stabilizes $\mathbb{F}_{q^\delta}V_0$,  where $q^\delta=q_0^s$ for some prime $s$, 
and $V_0:=V(2n,q_0)$ is an $\mathbb{F}_{q_0}$-subspace of $V$. 
Let $\mathcal C$ be a class of good elements,
suppose that $M\cap\cal C\ne\emptyset$ and let $t\in M\cap \cal C$. 
First we derive some basic facts about $r$, $M\cap\mathcal C$ and
$M$.

\begin{lemma}\label{c5cases}
The integer $n\geq3$,  $s\nmid n$ and $o(t)$ is divisible by 
a $\ppd(n;q^\delta)$ prime $r$ which is also a $\ppd(n;q_0)$ prime. Moreover the type $X'=X$, 
$M\cap\mathcal{C}$ is a single $M$-conjugacy class, and the number $c$ of $G$-conjugacy classes 
of subgroups $M$ is at most the quantity given in 
Table~$\ref{tab:c5}$. 
\end{lemma}

\begin{table}[h]
 \begin{center}
\begin{tabular}{l|c|l}\hline
$X$	 &$c\leq $          &Conditions on $s$\\ \hline
$\SL$	 &$q-1$ & \\ 
$\Su$	&$q+1$ & $s$ odd\\ 
$\Sp$   & $1$  & $s$ odd \\ 
$\So^\eps$&$1$&$s$ odd \\ \hline
\end{tabular}   
 \end{center}
\caption{$\mathbf{C}_5$ subgroups containing  good elements.}\label{tab:c5}
\end{table}

\begin{proof}
Suppose that $n=2$, so $o(t)=\frac{q+1}{(2,q-1)}$. We claim that 
$o(t)$ is not a $2$-power. Suppose to the contrary that it is a $2$-power, say $q=q_0^s=2^a-1>3$.
Hence $q_0^s\equiv 3\pmod{4}$, so $s$ is odd and
$2^a=q_0^s+1$ is divisible by $q_0+1$. This means that $q_0+1=2^b$, for some $b<a$, and 
$2^{a-b}=\frac{q_0^s+1}{q_0+1}=1-q_0+ \dots+q_0^{s-1}\equiv s\pmod{2^b}$,
which is a contradiction since $s$ is odd.    This proves the claim. 
Thus, $o(t)$ is divisible by a 
$\ppd(2s;q_0)$ prime, say $r$, and since $r\geq 2s+1\geq5$ it follows that $r$ divides
$|\SL_{4}(q_0)|$. This implies that $s=2$ and that all elements of $M$ of 
order $r$ lie in Singer cycles of $\SL_4(q_0)$ and hence are fixed point free on $V$,
which is a contradiction. Thus $n\geq3$, and since $q>2$, it follows that 
$o(t)$ is divisible by a $\ppd(n;q^\delta)$-prime $r$.

We note that $r\nmid |Z(G)|$. Therefore $r$ divides $|X'_{2n}(q_0)|$, 
and hence $r$ is a $\ppd(k;q_0)$ prime, for some $k \le 2n$. 
Now $r \mid (q_0^k-1)$, which implies that  
$r \mid (q^{\delta k}-1)$, so $k$ is a multiple of $n$. 
If $k=2n$, then an element $t_0$ of order $r$ in $\la t\ra$ lies 
in a Singer cycle of $\GL_{2n}(q_0)$
acting fixed point freely on $\mathbb{F}_{q^\delta}V_0$,
and hence $t$ has no fixed points in $V$, which is a contradiction.
Thus $r$ is a $\ppd(n;q_0)$ prime. Moreover, $s \nmid n$ since otherwise 
$q^{\delta n/s}-1 = q_0^n-1$ is divisible by $r$, which contradicts
the definition of $r$.   
 
By \cite[Table 4.5.A]{kleidman1990}, either $X'=X$ 
or (i) $X=\Su$, $s=2$, and $X'=\So^\eps$ (with $q$ odd) or $X'=\Sp$,
or (ii) $X=\So^\eps$ and $X'=\So^{\eps'}$, where $\eps=(\eps')^s$, or 
In case (i), $s=2$ and $r$ divides $q_0^n-1=q^{\delta n/2}-1$, 
which is a contradiction as in the previous paragraph.
In case (i), $n$ is even by Table~\ref{order:cent:good}, so $s$ is odd since $s\nmid n$. 
Thus $(\eps')^s=\eps'$ and so $\eps=\eps'$ and $X'=X$.
For the same reason $s$ is odd if $X=\Sp$; also $s$ is odd 
if $X=\Su$ by  \cite[Table 4.5.A]{kleidman1990}.

In each case $|N_M(\la t\ra):C_M(t)|=n$, and $N_M(\la t\ra)$ is transitive by conjugation on
the $\mathcal{C}\cap\la t\ra$, so $M\cap\mathcal C$ is a single $M$-conjugacy class. 
Finally the values of $c$ follow from \cite[Props. 4.5.3, 4.5.4, 4.5.10]{kleidman1990}.
\end{proof}

\begin{lemma}\label{lem:c5probs}
We have
\begin{enumerate}
 \item[(i)]
\[
p_5(\cnqx)<\begin{cases}
            8q^{-n^2+n+2}  &\mbox{if $X=\SL$}\\
            9q^{-4(n^2-n)/3+2}  &\mbox{if $X=\Su$}\\
            3 q^{-2(n^2-n)/3}  &\mbox{if $X=\Sp$}\\
            4 q^{-2(n^2-n)/3}  &\mbox{if $X=\So^\eps$}.\\
           \end{cases} 
\]
\item[(ii)]  $\tilde{p}_5(\cnqsp)\leq 3 q^{-2(n^2-n)/3}$.
\end{enumerate}
\end{lemma}

\begin{proof} \emph{Case $X=\SL$.} We refer to and use the notation of Lemma~\ref{c5cases}.
Here 
$M \cong (Z_{q-1} \circ \gl_{2n}(q_0))\cap \SL_{2n}(q)$,
and considering the scalars in $M$ we have $|M|=|\SL_{2n}(q_0)|\frac{(2n,q-1)}{(2n,q_0-1)}\leq 
|\SL_{2n}(q_0)|\frac{q-1}{q_0-1}$. Also $|C_M(t)| \geq (q_0^n-1)|\SL_{n}(q_0)|$. By Tables \ref{order:classical} and \ref{order:cent:good}, 
\[
 \frac{|M|}{|C_M(t)|^2}\leq \frac{(q-1)q_0^{4n^2}\Theta(1,2n;q_0)}
{(q_0^n-1)^2q_0^{2n^2}\Theta(1,n;q_0)^2}\leq \frac{4(q-1)q_0^{2n^2}}{(q_0^{n}-1)^2}
\]
since $\frac{\Theta(n+1,2n;q_0)}{\Theta(1,n;q_0)}\leq \frac{1}{1-q_0^{-1}-q_0^{-2}}\leq 4$, by Lemma~\ref{theta:ref}.
Applying Corollary~\ref{cor:cent:order} and noting that there are  $c\leq q-1$ classes of subgroups for each $s$, we see that 
the contribution to $p_5(\cnqsl)$ of subgroups with a fixed $s$ is at most 
\[
 c\frac{|M|}{|G|}\frac{|C_G(t)|^2}{|C_M(t)|^2}\leq \frac{4(q-1)(q^n-1)^2}{(q_0^n-1)^2q^{2n^2(1-1/s)-1}}.
\]
In particular, if $s=2$, then direct computation shows that this contribution is at most $\frac{8}{q^{n^2-n-2}}-
\frac{4}{q^{n^2-n-1}}$. For $s\geq3$, 
using the fact that  $\frac{(q^n-1)^2}{(q_0^n-1)^2}\leq 2q^{2n(1-1/s)}$ we see that the contribution for $s$ is less than
$\frac{8(q-1)}{q^{(2n^2-2n)(1-1/s)-1}}$. Now summing over all odd primes gives a contribution less than
\[
\frac{8(q-1)}{q^{4(n^2-n)/3-1}} \sum_{j\geq 0} q^{-j}< \frac{8}{q^{4(n^2-n)/3-2}},
\]
and combining with the contribution for $s=2$ yields
$p_5(\cnqsl)<8q^{-n^2+n+2}$.

\medskip\noindent
\emph{Case $X=\Su$.}
Here $|M| \leq |\Su_{2n}(q_0)|\frac{q+1}{q_0+1}$, and
$|C_M(t)|\geq(q_0^n+1)|\Su_n(q_0)|$. Using Tables \ref{order:classical} 
and \ref{order:cent:good}, 
\[
 \frac{|M|}{|C_M(t)|^2}\leq 
\frac{(q+1)q_0^{4n^2}\Theta(1,2n;-q_0)}{(q_0^{n}+1)^2q_0^{2n^2}\Theta(1,n;-q_0)^2}
\leq \frac{(q+1)q_0^{2n^2}}{(q_0^{n}+1)^2}
\]
since $\frac{\Theta(n+1,2n;-q_0)}{\Theta(1,n;-q_0)}<1$ by Lemma~\ref{theta:ref}.
Applying Corollary~\ref{cor:cent:order} and noting that there are  $c\leq q+1$ classes of subgroups for each $s$, we see that 
the contribution to $p_5(\cnqsu)$ of subgroups with a fixed $s$ is at most 
\[
 c\frac{|M|}{|G|}\frac{|C_G(t)|^2}{|C_M(t)|^2}\leq \frac{16(q+1)^2 (q^n+1)^2}{15(q_0^n+1)^2q^{2n^2(1-1/s)}}
<\frac{64}{15q^{(2n^2-2n)(1-1/s)-2}}
\]
using the fact that  $\frac{(q+1)(q^n+1)}{(q_0^n+1)}\leq 2q^{n(1-1/s)+1}$. Now we add over $s$, 
and note that $s\geq3$ since $n$ is odd (Table~\ref{order:cent:good}). A cruder estimate that the $\SL$-case gives
\[
p_5(\cnqsu)<\frac{64}{15q^{4(n^2-n)/3-2}} \sum_{j\geq 0} q^{-j}< \frac{9}{q^{4(n^2-n)/3-2}}.
\]

\medskip\noindent
\emph{Case $X=\Sp$ and part (ii).}
Here $s$ is odd, $M=\Sp_{2n}(q_0)$, and
$|C_M(t)|=(q_0^{n/2}+1)|\Sp_n(q_0)|$. Using Tables \ref{order:classical} 
and \ref{order:cent:good}, 
\[
 \frac{|M|}{|C_M(t)|^2}= 
\frac{q_0^{2n^2+n}\Theta(1,n;q_0^2)}{(q_0^{n/2}+1)^2q_0^{n^2+n}\Theta(1,n/2;q_0^2)^2}
\leq \frac{16q_0^{n^2-n}}{11}
\]
since by Lemma \ref{theta:ref}, $\frac{\Theta(n/2+1,n;q_0^2)}{\Theta(1,n/2;q_0^2)}< 16/11$ (note $q_0^2\geq4$). 
Applying Corollary~\ref{cor:cent:order} and noting that there is a unique class of subgroups for each $s$, we see that 
the contribution to $p_5(\cnqsp)$ from subgroups with a fixed $s$ is at most 
\[
\frac{|M|}{|G|}\frac{|C_G(t)|^2}{|C_M(t)|^2}\leq \frac{25}{11q^{(n^2-n)(1-1/s)}}
<\frac{2.3}{q^{(n^2-n)(1-1/s)}}.
\]
Adding over $s\geq3$, we obtain 
$p_5(\cnqsp)<2.3\, q^{-2(n^2-n)/3} \sum_{j\geq 0} q^{-j}$, and since $\sum_{j\geq 0} q^{-j}=1/(1-q^{-1})\leq 8/7$
(since $q\geq8$),  this is less than $3q^{-2(n^2-n)/3}$.
We note here that, by Lemma~\ref{c5cases}, if $q$ is even, then every maximal $\mathbf{C}_5$-subgroup of 
$\So^\eps_{2n}(q)$ containing good elements is contained in a maximal $\mathbf{C}_5$-subgroup of
$\Sp_{2n}(q)$, and it follows that 
$\tilde{p}_5(\cnqsp)<3q^{-2(n^2-n)/3}$, as in part (ii).

\medskip\noindent
\emph{Case $X=\So^\eps$.}
Again $s$ is odd, and here $M=\So^\eps_{2n}(q_0)$, and
$|C_M(t)|=(q_0^{n/2}+1)|\So^{-\eps}_n(q_0)|$. Using Tables \ref{order:classical} 
and \ref{order:cent:good}, 
\[
 \frac{|M|}{|C_M(t)|^2}= 
\frac{q_0^{2n^2-n}(1-\eps q_0^{-n})\Theta(1,n-1;q_0^2)}{(q_0^{n/2}+1)^2 q_0^{n^2-n}(1+\eps q_0^{-n/2})^2\Theta(1,n/2-1;q_0^2)^2}
\leq \frac{17}{9}\cdot \frac{16q_0^{n^2-n}}{11}
\]
since by Lemma \ref{theta:ref}, $\frac{\Theta(n/2,n-1;q_0^2)}{\Theta(1,n/2-1;q_0^2)}< 16/11$ (note $q_0^2\geq4$). 
By Corollary~\ref{cor:cent:order}, and using the inequality \eqref{cor-so}, since $q=q_0^s\geq 8$ we have 
$|C_G(t)|^2/|G|< 1.1 q^{-n^2+n}$.  Noting that there is a unique class of subgroups for each $s$, we see that 
the contribution to $p_5(\cnqsog)$ of subgroups with a fixed $s$ is at most 
\[
\frac{|M|}{|G|}\frac{|C_G(t)|^2}{|C_M(t)|^2}\leq \frac{17\times 16\times1.1}{99\,q^{(n^2-n)(1-1/s)}}
<\frac{3.1}{q^{(n^2-n)(1-1/s)}}.
\]
Adding over $s\geq3$, as in the $\Sp$-case we obtain 
$p_5(\cnqsog)<4 q^{-2(n^2-n)/3}$.
\end{proof}

\section{{$\mathbf{C}_6$:} Normalizers of symplectic-type groups}
\label{c6}

A maximal subgroup $M$ of $G=X_{2n}(q)$ belonging to Aschbacher class $\mathbf{C}_6$, 
is contained in $(Z(G)\circ R).\Sp_{2d}(2)$, where $R=2^{1+2d}$ is an extraspecial $2$-group.
Moreover $2n=2^d$, $q$ is odd, and $R$ is absolutely irreducible on $V$  (see \cite[Section 4.6, 
especially Table 4.6.B]{kleidman1990}). 
We deal with all cases where a good element is not a $2$-element.
It turns out that, with a single exception, these maximal goups do not contain good elements.

\begin{lemma}\label{lem:c6} Suppose that $(n,q)\ne (2,2^a-1)$ for any $a$. Then 
either 
\begin{enumerate}
 \item[(i)]  $p_6(\cnqx)=\tilde{p}_6(\cnqsp)=0$, or
\item[(ii)] $(X,n,q)=(\SL,2,5)$, $M=(Z_4\circ 2^{1+6}).\Sp_{6}(2)'$ contains good elements of order $3$,
and $p_6(\SL,2,5)<\frac{1}{48}<0.021$.
\end{enumerate}
\end{lemma}

\begin{proof}
Suppose that $M$, as above, contains a good element $t$, and let $V=U\oplus W$ be as in 
Lemma~\ref{good:cent:order}. Let $r$ be a $\ppd(n,q^\delta)$-prime dividing $o(t)$ if $n\geq3$, 
or $r=o(t)=(q+1)/2$ if $n=2$.   
By the conditions on $n, q$, the integer $r$ is not a power of $2$.

Since $R$ is irreducible on $V$, for each nonzero $v \in V$, the $R$-orbit $v^R$ spans $V$ and in particular $|v^R|\geq 2n$.
Also $|v^R|$ divides $|R|=2^{1+2d}=8n^2$. It follows that $r$ does not divide $|v^R|$, 
and so $v^R$ contains a fixed vector of $t$, that is, $U\cap v^R\ne\emptyset$.
Thus the number of $R$-orbits on non-zero vectors is at most $|U|-1 = q^{\delta n}-1$.
However each of these $R$-orbits has length at most $|R|$, and so 
$q^{2\delta n}-1=|V|-1 \leq 8 n^2 (q^{\delta n}-1)$. Hence $3^{\delta n}<q^{\delta n}+1\leq 8n^2$,
which implies that $(n,q^\delta, r)=(4,3,5), (2,5,3)$ (since $(n,q)=(2,3)$ is excluded from the analysis). 

If $(n,q^\delta, r)=(4,3,5)$, then $M$ must contain a good element $t_0$ of order $5$ (some power of $t$).
Since $r=5=2^2+1$, $\la t_0\ra$ projects to a cyclic torus of order $5$ of $\Sp_6(2)$, and hence $t_0$ leaves 
invariant a tensor decomposition $V=V_1\otimes V_2$ with $\dim(V_1)=2, \dim(V_2)=4$, contradicting Lemma~\ref{lem:c4probs}.

Thus $(n,q^\delta, r)=(2,5,3)$. Here $X=\SL$, and the subgroup $M$ of $\SL_4(5)$ was constructed and
examined using {\sf{GAP}} \cite{GAP}. The group $M$ has two conjugacy classes of elements of order $3$,
and one of them consists of good elements (the other contains fixed point free elements). 
This computation showed that $\frac{|G|}{|M|}\frac{|M\cap\mathcal C|^2}{|\mathcal C|^2}=\frac{256}{50375}<\frac{1}{196}$.
By \cite[Prop. 4.6.5]{kleidman1990}, there are four $G$-conjugacy classes of such maximal subgroups $M$,
and hence $p_6(\SL,2,5)<\frac{1}{48}$.
\end{proof}

\section{{$\mathbf{C}_7$:} Stabilizers of tensor powers} 
\label{c7}

A maximal subgroup $M$ of $X_{2n}(q)$ belonging to the Aschbacher class 
$\mathbf{C}_7$ is the stabilizer of a tensor power decomposition $V=
V_1\otimes\dots\otimes V_\ell$, where $\dim(V_i)=d$ for all $i$ and $2n=d^\ell$, 
for some $\ell\geq2$; and so $M$ is contained in a group of the shape
$(\gl_{d}(q)\wr S_\ell) \cap \SL_{2n}(q)$ or $(\gu_{d}(q)
\wr S_\ell) \cap \Su_{2n}(q)$. We show that $M$ contains no good elements if $n\geq3$. 

\begin{lemma}\label{lem:c7}
$p_7(\cnqx)=0$, for any $X, q$, and any $n\geq3$, and $\tilde{p}_7(\cnqsp)=0$. 
\end{lemma}

\begin{proof}
Suppose to the contrary that $M$, as above, contains a good element $t$, and let $V=U\oplus W$ be as in 
Lemma~\ref{good:cent:order}. Let $r$ be a $\ppd(n,q^\delta)$-prime dividing $o(t)$,
and note in particular that $r\geq n+1$. Since $2^n \ge 2n=d^\ell \ge 2^\ell$, we 
obtain $n \ge \ell$, and in particular $r>\ell$. 
Hence $r$ divides the order of an element in $\gl_{d}(q)$ or $\gu_{d}(q)$. 
By definition of $r$, this implies that  $d \ge n$, and the only 
possibility is $n=d=\ell=2$, which contradicts $n\geq3$. 
\end{proof}

\section{{$\mathbf{C}_8$:} Classical subgroups}
\label{c8}

By \cite[Table 4.8.A]{kleidman1990}, if $G=X_{2n}(q)$ has maximal subgroups belonging to $\mathbf{C}_8$ 
then either $X=\SL$, or $q$ is even, $X=\Sp$, and $M \cong  \So_{2n}^\eps(q)$.
We do not consider the latter case for estimating $p_8(\cnqx)$ since we assume that $q$ is odd if $X=\Sp$. 
Also we note that $\tilde{p}_8(\cnqsp)=0$ since the subgroups $\So^\eps_{2n}(q)$
have no proper maximal $\mathbf{C}_8$ subgroups. We establish the following 
estimates.

\begin{lemma}\label{lem:c8probs}
\begin{enumerate}
 \item[(i)] $p_8(\cnqx)=0$ if $X=\Su$ or $\So^\eps$, or if $(X,q)=(\Sp, odd)$, and 
$p_8(\cnqsl)<2.5 q^{-n^2+n+2}$.
\item[(ii)]  $\tilde{p}_8(\cnqsp)=0$.
\end{enumerate}
\end{lemma}

\begin{proof} Part (ii) follows from our comments above. Also, for part (i), by our comments above, 
the only case to consider is $X=\SL$, so let $G=\SL_{2n}(q)$.
Suppose that $M$ is a maximal $\mathbf{C}_8$-subgroup of $G$ containing a good element $t$ from the 
$G$-conjugacy class $\mathcal C$, and that there are $c$ conjugacy classes of such subgroups $M$. 
Then, using the results in \cite[Section 4.8]{kleidman1990}, especially 
\cite[Table 4.8.A]{kleidman1990}, and Lemma~\ref{good:cent:order}, the cases we need to consider are

\begin{tabular}{lll}
(a) &$M=Z_{(2n,q-1)}\circ\Sp_{2n}(q)$,       &$c= (q-1,n)$, and $n$ is even; \\
(b) &$M=Z_{(2n,q-1)}\circ\Su_{2n}(q_0)$, &$c\leq q-1$, $q=q_0^2$, and $n$ is odd;\\
(c) &$M=Z_{(2n,q-1)}\circ\So^\eps_{2n}(q)$,  &$c\leq q-1$, $q$ is odd, and $n$ is even.\\
\end{tabular}

Moreover, in all cases the $n$ conjugates of $t$ in the cyclic torus $T$ of 
Lemma~\ref{good:cent:order} are still $N_M(T)$-conjugate, and hence 
$M\cap\mathcal C$ is a single $M$-conjugacy class. 
Thus by \eqref{prob:cc} and \eqref{c1:singleclass}, the contribution to 
$p_8(\cnqsl)$ from each type of subgroup $M$ is at most 
$\frac{(q-1)|M|}{|G|}\frac{|C_G(t)|^2}{|C_M(t)|^2}$, 
which by Corollary~\ref{cor:cent:order}
is at most $\frac{(q^n-1)^2}{q^{2n^2-1}}\frac{|M|}{|C_M(t)|^2}$.

\medskip\noindent\emph{Case (i)(a).} 
By Tables \ref{order:classical} and 
\ref{order:cent:good}, and noting that $|M|\leq (q-1)|\Sp_{2n}(q)|$ and that   $\frac{\Theta(n/2+1,n;q^2)}{\Theta(1,n/2;q^2)}\leq 
(1-q^{-2}-q^{-4})^{-1}\leq \frac{16}{11}$  (or less than $\frac{9}{8}$ if $q$ is odd), by Lemma~\ref{theta:ref}, the contribution is at most
\begin{eqnarray*}
\frac{(q^n-1)^2}{q^{2n^2-1}} \frac{(q-1)q^{2n^2+n}\Theta(1,n;q^2)}{(q^{n/2}+1)^2q^{n^2+n}\Theta(1,n/2;q^2)^2}
&<&\frac{16(q^{n/2}-1)^2}{11q^{n^2-2}}<\frac{16}{11 q^{n^2-n-2}}
\end{eqnarray*}
or less than $\frac{9}{8 q^{n^2-n-2}}$ if $q$ is odd.

\medskip\noindent\emph{Case (i)(b).} 
By Tables \ref{order:classical} and 
\ref{order:cent:good}, and noting that $|M|\leq (q_0-1)|\Su_{2n}(q_0)|$ and that  
$\frac{\Theta(n+1,2n;-q_0)}{\Theta(1,n;-q_0)}<1$, by Lemma~\ref{theta:ref}, the contribution is at most
\begin{eqnarray*}
\frac{(q^n-1)^2}{q^{2n^2-1}} \frac{(q-1)q_0^{4n^2}\Theta(1,2n;-q_0)}{(q_0^{n}+1)^2q_0^{2n^2}\Theta(1,n;-q_0)^2}
&<&\frac{1}{q^{n^2-n-2}}.
\end{eqnarray*}

\medskip\noindent\emph{Case (i)(c).} 
By Tables \ref{order:classical} and 
\ref{order:cent:good}, and noting that $|M|\leq (q-1)|\So^\eps_{2n}(q)|$, and (using Lemma~\ref{theta:ref}(iii)) that   
$\frac{\Theta(n/2,n-1;q^2)}{\Theta(1,n/2-1;q^2)}\leq 
(1-q^{-2}-q^{-4})^{-1}< \frac{9}{8}$ (since $q$ is odd), the contribution for type $\eps$ is at most
\begin{eqnarray*}
{}&&\frac{(q^n-1)^2}{q^{2n^2-1}} \frac{(q-1)q^{2n^2-n}(1-\eps q^{-n})\Theta(1,n-1;q^2)}{(q^{n/2}+1)^2q^{n^2-n}(1+\eps q^{-n/2})^2\Theta(1,n/2-1;q^2)^2}\\
&\leq& \frac{9(q^{n/2}-1)^2(1-\eps q^{-n})}{8q^{n^2-2}(1+\eps q^{-n/2})^2}       
\leq \frac{9(q^n-\eps)}{8q^{n^2-2}}
\end{eqnarray*}
and hence, summing over $\eps$, the contribution in this case is at most $\frac{9}{4q^{n^2-n-2}}$. 

Collecting these results we see that $p_8(\cnqsl)\leq kq^{-n^2+n+2}$, where $k=1$ if $n$ is odd, 
$k=\frac{16}{11}$ if both $n$ and $q$ are even, and $k=\frac{27}{8}$ if $n$ is even and $q$ is odd.
\end{proof}

\section{{$\mathbf{C}_9$:} Nearly simple subgroups} 
\label{c9}

Let $G=X_{2n}(q)$ and let $M$ be a maximal subgroup of $G$ belonging to the Aschbacher Class $\mathbf{C}_9$
for $G$, or if $(X,q)=(\Sp, even)$, a maximal subgroup in Aschbacher Class $\mathbf{C}_9$
for $\So^\eps_{2n}(q)$. 
By \cite[Theorem 4.2]{liebeck85.426}, one of the following holds.
\begin{itemize}
\item[(i)] $M\cong Z\times S_\ell$ where $\ell = 2n+1$ or $2n+2$, $Z=Z(G)\leq Z_2$, and 
$X=\So^\eps$ or $\Sp$ as in [1.]--[4.] of Section~\ref{ansn}; 
\item[(ii)] $|M| < q^{6\delta n}$.
\end{itemize}

We obtain the following estimates for the $\mathbf{C}_9$-probabilities
(which again are reasonable for $n$ large).

\begin{lemma}\label{lem:c9probs} For $n\geq9$,
\begin{enumerate}
 \item[(a)] $p_9(\cnqx)<6q^{-2n^2+10.6n}$ if $X=\SL$;
 \item[(b)] $p_9(\cnqx)<6q^{-2n^2+16.6n}$ if $X=\Su$;
\item[(c)] $p_9(\cnqx)<9q^{-n^2+9.6n}$ if $X=\Sp$ (with $q$ odd) or $\So^\eps$; 
\item[(d)]  $\tilde{p}_9(\cnqsp)<9q^{-n^2+9.6n}$ for $q$ even.
\end{enumerate}
\end{lemma}

\begin{proof}
Let $G=X_{2n}(q)$ and let $M$ be a maximal $\mathbf{C}_9$-subgroup of $G$, or of 
$\So^\eps_{2n}(q)$ if $X=\Sp$ and $q$ is even.
Suppose that  $\mathcal{C}$ is a conjugacy class of good elements with $M\cap\cal C\ne\emptyset$.  
By Lemma~\ref{lem:c9-altprobs}, the contributions from groups in case (i) to the 
$\mathbf{C}_9$-probabilities satisfy $p^{(i)}_9(\cnqx)<q^{-n^2+4n+3}$ and
 $\tilde{p}^{(i)}_9(\cnqsp)<q^{-n^2+4n+3}$.

We deal now with the groups in case (ii), using the following upper bound
for the number $c(G)$ of conjugacy classes of $\mathbf{C}_9$-subgroups in $G$ 
proved in \cite[Theorem 1.1]{hasa}: 
\begin{eqnarray}
c(G) < N(n,q):= 2(2n)^{5.2}+2n\log_2\log_2q.
\end{eqnarray}
Let  $\cal S$ be the set of $G$-conjugacy classes of maximal $\mathbf{C}_9$-subgroups in case (ii).
For these groups $M$, we use 
the trivial upper bound $|M\cap \mathcal{C}|\leq |M| < q^{6\delta n}$, so, by \eqref{prob:cc}, the contribution $p_9^{(ii)}(\cnqx)$ to 
$p_9(\cnqx)$ from all groups in case (ii) is at most 

\begin{eqnarray}\label{c9:general}
p_9^{(ii)}(\cnqx)&\leq&\sum_{\mathbf{S}\in\cal S}\frac{|G|}{|M(\mathbf{S})|}
\frac{|M(\mathbf{S})\cap{\cal C}|^2}{|{\cal C}|^2}\leq 
\frac{|G|}{|{\cal C}|^2}\sum_{\mathbf{S}\in\cal S}|M(\mathbf{S})|\nonumber \\
&\leq& \frac{|C_G(t)|^2}{|G|} N(n,q) q^{6\delta n}.
\end{eqnarray}
Note that, for all $n\geq9$ and $q\geq2$, we have $2n\leq q^{n/2}$  so 
$2(2n)^{5.2}\leq 2q^{2.6n}$ and  $2n\log_2\log_2q\leq q^{n}$.
Thus $N(n,q)< 3q^{2.6n}$.
We use the upper bounds from Corollary~\ref{cor:cent:order}. If $X=\SL$ or $\Su$, then
$\frac{|C_G(t)|^2}{|G|}\leq 2 q^{-2n^2+2n}$, and so
\[
p_9(\cnqsl) = p_9^{(ii)}(\cnqsl)< \frac{6}{q^{2n^2-10.6n}} 
\]
and
\[
p_9(\cnqsu) =p_9^{(ii)}(\cnqsu)< \frac{6}{q^{2n^2-16.6n}}. 
\]
If $X=\Sp$ or $\So^\eps$, then
$\frac{|C_G(t)|^2}{|G|}\leq \frac{25}{9} q^{-n^2+n}$, and so
\[
p_9^{(ii)}(\cnqx)< \frac{25}{3q^{n^2-9.6n}}. 
\]
Thus, for $X=\Sp, \So^\eps$, adding to the estimate for the groups in case (i) we have 
$p_9(\cnqx)=p_9^{(i)}(\cnqx)+p_9^{(ii)}(\cnqx)  <9q^{-n^2+9.6n}$.
Similarly $\tilde{p}_9^{(ii)}(\cnqsp)$ is less than $9q^{-n^2+9.6n}$.
\end{proof}

\section{Proofs of Theorems~\ref{broadbrush} and~\ref{main2}}\label{sec:proof}

We draw together results of the previous sections, namely Lemmas~\ref{lem:c2probs},~\ref{lem:c3probs},~\ref{lem:c4probs}, \ref{lem:c5probs},
~\ref{lem:c6},~\ref{lem:c7},~\ref{lem:c8probs},~\ref{lem:c9probs}. Although most of the results hold for smaller $n$, 
let us assume that $n\geq 9$ for $X=\SL, \Su$, and let $n\ge10$ for $X=\Sp, \So^\eps$. Then, using
\eqref{third:prob} and \eqref{main3:prob}, the probabilities $p(\cnqx)\leq\sum_{i=2}^9p_i(\cnqx)$, 
$\tilde{p}(\cnqsp)\leq\sum_{i=2}^9\tilde{p}_i(\cnqx)$, and upper bounds for the 
$p_i(\cnqx)$ and $\tilde{p}_i(\cnqsp)$ are obtained in the results mentioned above.
We summarise these in Tables~\ref{tab:proof}. We give in Table \ref{tab:conditions} a summary, for each $i$, of the conditions on $n, q$ needed for the upper bounds on $p_i(\cnqx)$ and $\tilde{p}_i(\cnqsp)$ to hold. Recall that throughout the paper we assume that  $n\geq 2$ if $X=\SL,\Su$ or $\Sp$ and $n\geq 4$ if $X=\So^\eps$. Moreover, if $i=5$, then $n\geq 3$ by Lemma \ref{c5cases}.

\begin{table}[h]
 \begin{center}
{\footnotesize
\begin{tabular}{l|l|l|l|l|l}\hline
$i$	 &$p_i(\cnqsl)$          &$p_i(\cnqsu)$    &$p_i(\cnqsp)$       &$p_i(\cnqsog)$        &$\tilde{p}_i(\cnqsp)$ \\ \hline \hline
$2$	 &$q^{-n^2}$             &$0$              &$0$                 &$q^{-n^2+3n}/2$        &$0$ \\ 
$3$	 &$3.6q^{-n^2}$          &$5q^{-4n^2/3}$   &$3.7q^{-n^2/2}$     &$10.6q^{-n^2/2}$&$3.7q^{-n^2/2}$ \\ 
$4$	 &$0$                    &$0$                  &$0$                 &$0$                   &$0$ \\ 
$5$	 &$8q^{-n^2+n+2}$     &$9q^{-4(n^2-n)/3+2}$&$3q^{-2(n^2-n)/3}$&$4q^{-2(n^2-n)/3}$&$3q^{-2(n^2-n)/3}$ \\ 
$6$	 &$1/48$                  &$0$            &$0$          &$0$         &$0$ \\ 
$7$	 &$0$                    &$0$            &$0$          &$0$         &$0$ \\ 
$8$	 &$2.5q^{-n^2+n+2}$      &$0$            &$0$          &$0$         &$0$ \\ 
$9$      &$6q^{-2n^2+10.6n}$  &$6q^{-2n^2+16.6n}$&$9q^{-n^2+9.6n}$&$9q^{-n^2+9.6n}$   &$9q^{-n^2+9.6n}$\\ \hline                        
\end{tabular}   
} 
\end{center}
\caption{Upper bounds for $p_i(\cnqx)$ and $\tilde{p}_i(\cnqsp)$.}\label{tab:proof}
\end{table}

\begin{table}[h]
 \begin{center}
{\footnotesize
\begin{tabular}{l|l}\hline
$i$	 &$p_i(\cnqx)$ and $\tilde{p}_i(\cnqsp)$ \\ \hline \hline
$1$   & $q\neq 2$ if $X=\SL$ or $\Sp$\\
$2$	 &$(n,q)\neq (2,5), (2,2^a-1)$  \\ 
$3$	 &  $(n,q)\neq (2,3), (2,7)$    \\ 
$4$	&  $n\geq 3$       \\ 
$5$&  $n\geq 3$         \\ 
$6$	 &  $(n,q)\neq (2,2^a-1)$    \\ 
$7$	 &  $n\geq 3$         \\ 
$8$	 & $n\geq 3$ if $X=\SL$   \\ 
$9$      &$n\geq 9$  \\ \hline                        
\end{tabular}   
} 
\end{center}
\caption{Conditions for $n,q$ for the upper bounds for $p_i(\cnqx)$ and $\tilde{p}_i(\cnqsp)$.}\label{tab:conditions}
\end{table}

Theorem~\ref{main2} follows immediately from Tables \ref{tab:proof}, \ref{tab:conditions}. To prove Theorem~\ref{broadbrush}, we note that, 
Theorem~\ref{main} gives a (large) non-zero probability $p_1$ that $\la t, t'\ra$ is irreducible,
where $t$ is a good element in $G=X_{2n}(q)$ and $t'$ is a random conjugate of $t$ (where we take $q\geq4$ if $X=\SL$). 
Then, adding the terms in the appropriate column of Table~\ref{tab:proof} gives a  probability $p_2$ (smaller than $p_1$ if, say, $n\geq20$) that
$\la t, t'\ra$ is irreducible and contained in a maximal subgroup of $G$, or in the case $G=\Sp_{2n}(q)$ with $q$ even,
contained in a maximal subgroup of a subgroup $\So^\eps(q)$. Thus for $n\geq20$, say, this gives a positive constant $c$ , at most $p_1-p_2$,
as a lower bound for the probability that $\la t,t'\ra$ is equal to $G$ or, if $G=\Sp_{2n}(q)$ with $q$ even,
is equal to a subgroup $\So^\eps(q)$. The constant $c$ in Theorem~\ref{broadbrush} must be a modification of this to
take account of the finite number of dimensions $n$ where the contribution 
for $p_9(\cnqx)$ is too large for this argument to work - it should be the minimum of $p_1-p_2$, and the actual proportions for this finite number of cases.


\begin{thebibliography}{10}

\bibitem{aschbacher84.469}
M.~Aschbacher, \emph{On the maximal subgroups of the finite classical groups},
  Invent. Math. \textbf{76} (1984), no.~3, 469--514.

\bibitem{atlas}
J. H. Conway, R. T. Curtis, S. P. Norton, R. A. Parker, \emph{Atlas of finite groups. 
Maximal subgroups and ordinary characters for simple groups. With computational assistance from J. G. Thackray},
Oxford University Press, 1985.

\bibitem{DLLO}
H. Dietrich, C. R. Leedham-Green, F. L¨ubeck, and E. A. O’Brien,
Constructive recognition of classical groups in even characteristic,
\emph{J. Algebra} {\bf391} (2013), 227--255.

\bibitem{GAP}
The {\sf GAP}-Group, \emph{{\sf GAP} -- Groups, Algorithms, and Programming, 
Version 4.7.2}; 2013, \verb+(http://www.gap-system.org)+.

\bibitem{hasa}
J. H\"as\"a, Growth of cross-characteristic representations of finite
quasisimple groups of Lie type, Submitted, 2012.

\bibitem{Hu}
B. Huppert, \emph{Endliche Gruppen I}, Springer, Berlin, 1967.

\bibitem{braueratlas}
C. Jansen, K. Lux, R. Parker, R. Wilson, \emph{An atlas of Brauer 
characters}, London Mathematical Society Monographs volume 11, 
Clarendon Press, Oxford, 1995. 

\bibitem{kleidman87.173}
P.~Kleidman, \emph{The maximal subgroups of the finite {$8$}-dimensional
  orthogonal groups {$P\Omega^+_8(q)$} and of their automorphism groups}, J.
  Algebra \textbf{110} (1987), no.~1, 173--242.

\bibitem{kleidman1990}
P.~Kleidman and M.~Liebeck, \emph{The subgroup structure of the finite
  classical groups}, London Mathematical Society Lecture Note Series, vol. 129,
  Cambridge University Press, Cambridge, 1990.

\bibitem{liebeck85.426}
M.~W. Liebeck, \emph{On the orders of maximal subgroups of the finite classical
  groups}, Proc. London Math. Soc. (3) \textbf{50} (1985), no.~3, 426--446.

\bibitem{NP}
P. M. Neumann and C. E. Praeger, Cyclic matrices over finite fields, 
\emph{J. London Math. Soc.} {\bf52} (1995), 263--284.

\bibitem{recog}
M. Neunh\"oeffer, {\sf GAP} Package recog, Version 1.2. 28 May, 2012.
 \verb+(http://www-groups.mcs.st-and.ac.uk/~neunhoef/Computer/Software/Gap/recog.html)+.

\bibitem{NS}
M. Neunhoeffer and \'{A}. Seress. Constructive resognition of $\SL_n(q)$.
In preparation.

\bibitem{prastrong}
C.~Praeger and {\'A}~Seress, \emph{Probabilistic generation of finite classical
  groups in odd characteristic by involutions}, J. Group Theory {\bf 14} (2011), 521-545.

\bibitem{zsigmondy92.265}
K.~Zsigmondy, \emph{Zur {T}heorie der {P}otenzreste}, Monatsh. Math. Phys.
  \textbf{3} (1892), no.~1, 265--284.

\end{thebibliography}
\end{document}